\documentclass[a4paper]{article}

\usepackage{amsmath, amsfonts, amsthm, mathrsfs, amssymb}
\usepackage{graphicx, subfigure, threeparttable, booktabs, tabularx}
\usepackage{xcolor}
\usepackage{bm}
\usepackage{hyperref}
\usepackage{float}
\usepackage[a4paper, left=2.5cm,right=2.5cm,top=2.5cm,bottom=3cm]{geometry}

\newtheorem{theorem}{Theorem}[section]

\newtheorem{lemma}{Lemma}[section]

\newtheorem{remark}{Remark}[section]

\newtheorem{example}{Example}
\newtheorem{assumption}{Assumption}[section]
\allowdisplaybreaks[4]

\usepackage{authblk}

\begin{document}

\title{Dimension-independent convergence rate of propagation of chaos and numerical analysis for McKean-Vlasov stochastic differential equations with coefficients nonlinearly dependent on measure}

\author[a,b]{Yuhang Zhang}
\author[c]{Minghui Song\thanks{Corresponding author: songmh@hit.edu.cn}}
\affil[a]{School of Astronautics, Harbin Institute of Technology, Harbin, China}
\affil[b]{Zhengzhou Research Institute, Harbin Institute of Technology, Zhengzhou, China}
\affil[c]{School of Mathematics, Harbin Institute of Technology, Harbin, China}

\date{}

\maketitle

\begin{abstract}
In contrast to ordinary stochastic differential equations (SDEs), the numerical simulation of McKean-Vlasov stochastic differential equations (MV-SDEs) requires approximating the distribution law first. Based on the theory of propagation of chaos (PoC), particle approximation method is widely used. Then, a natural question is to investigate the convergence rate of the method (also referred to as the convergence rate of PoC). In fact, the PoC convergence rate is well understood for MV-SDEs with coefficients linearly dependent on the measure, but the rate deteriorates with dimension $d$ under the $L^p$-Wasserstein metric for nonlinear measure-dependent coefficients, even when Lipschitz continuity with respect to the measure is assumed. The main objective of this paper is to establish a dimension-independent convergence result of PoC for MV-SDEs whose coefficients are nonlinear with respect to the measure component but Lipschitz continuous. As a complement we further give the time discretization of the equations and thus verify the convergence rate of PoC using numerical experiments.\\

\end{abstract}


\section{Introduction}
This paper establishes the convergence rate of propagation of chaos (PoC) under the $L^p$-Wasserstein metric and the $L^p$ convergence rate of the tamed Euler-Maruyama (EM) method to the McKean-Vlasov stochastic differential equations (MV-SDEs) with super-linear growth in the spatial component in the drift and Lipschitz continuity in the measure component.  MV-SDEs refer to the whole family of SDEs whose coefficients depend not only on the state components but also on their probability distributions. These equations were first proposed by H. McKean as a stochastic model naturally associated with a class of nonlinear parabolic equations (\cite{McKean1966,McKean1967}), and are now widely studied due to their important applications in areas such as statistical physics (see e.g. \cite{Martzel2001,Kolokoltsov2010}), mathematical biology (\cite{Baladron2012,Bossy2015}), mathematical finance(\cite{J.Zhang2021}), among others.

MV-SDEs, being clearly more involved than the classical It\^o’s SDEs, let us first introduce some studies on the well-posedness of the exact solution of such equations. Strong existence and uniqueness results under global Lipschitz conditions were established via fixed-point theorems (\cite{Sznitman1991book, Carmona2016, Bahlali2020}). Subsequent work extended these results to cases with one-sided Lipschitz drift and globally Lipschitz diffusion coefficients (\cite{G.dosReis2019} and \cite{WangFengyu2018}). \cite{Kumar2022} proved existence and uniqueness for MV-SDEs with superlinear growth in both drift and diffusion coefficients. Further studies on well-posedness can be found in \cite{Bauer2018, Chaudru2020, Mishura2020, FukeWu2023,Xingyuan Chen2023Arxiv} and related references.

Despite theoretical guarantees on well-posedness, explicit solutions for MV-SDEs with superlinear coefficients remain elusive, necessitating numerical methods. The numerical framework McKean-Vlasov equations were initially proposed in \cite{Bossy1997} and have been investigated further in a number of more recent works. Ref. \cite{Bossy1997} suggests that the numerical discretization of MV-SDEs should involve two steps: first, approximating the true measures in the coefficients with empirical measures; second, time discretization of the resulting particle system. To be more specific, consider the MV-SDEs of the form
\begin{align}\label{MVSDEs*}
       {\rm d}x_t=f(x_t,\mu_t^x){\rm d}t+g(x_t,\mu_t^x){\rm d}W_t, ~t\in[0,T], 
\end{align}
with the initial data $x_0$, where $\mu_t^x$ denotes the law of the process $x$ at time $t$, i.e., $\mu_t^x=\mathbb{P}\circ x_t^{-1}$. To simulate this system, let $(x_0^i,W^i)_{1\le i\le N}$ be $N$ independent copies of $(x_0,W)$, each generating a particle through 
\begin{align}\label{NIPS*}
{\rm d}x_t^i= f(x_t^i,\mu_t^{x^i}) {\rm d}t+ g(x_t^i,\mu_t^{x^i}){\rm d}W^i_t.
\end{align}
Let these particles interact with each other through their empirical measure, i.e., introduce $X_t^{i,N}$, the solution of the following interacting particle system 
\begin{align}\label{IPS*}
{\rm d}X_t^{i,N}= f(X_t^{i,N},\hat\mu_t^{X,N}) {\rm d}t+ g(X_t^{i,N},\hat\mu_t^{X,N}){\rm d}W^i_t, \quad X_0^{i,N}=x_0^i,
\end{align}
where $\hat\mu_t^{X,N}=\frac{1}{N}\sum_{j=1}^N\delta_{X_t^{j,N}}$ is the empirical measure. As shown in \cite{Bossy1997}, the empirical distribution $\hat\mu_t^{X,N}$ converges to $\mu_t^{x^i}$ (and by uniqueness $\mu_t^{x^i}=\mu_t^{x}$), enabling approximation of the true measure. With this particle approximation in hand, discretizing \eqref{IPS*} in time yields a numerical solution for the MV-SDEs \eqref{MVSDEs*}.

Along this line, analyzing the convergence of numerical solutions to \eqref{MVSDEs*} requires studying the PoC rate, i.e., the convergence of the interacting particle system \eqref{IPS*} to the non-interacting limit\eqref{NIPS*}. While Wasserstein convergence rates for linear functionals of measures are well understood, nonlinear functionals often exhibit dimension-dependent rates under Wasserstein metrics (\cite{Carmona2018, Fournier2015, Guo2024ESAIM, ShuaibinGao2024,Szpruch2021}). However, some recent work suggest that a dimension-independent $\mathcal{O}(N^{-1/2})$ rate is achievable for specific classes of MV-SDEs (\cite{F.Delarue2019,S.Meleard1996,Szpruch2021}). Moreover, the dimension-independent $\mathcal{O}(N^{-1/2})$ rate is estimated numerically in \cite{Xingyuan Chen2023Arxiv}, and the authors highlight gaps in theoretical results, underscoring the need for further study. In this paper, we propose a specific class of MV-SDEs whose coefficients nonlinearly dependent on the measure, and prove a dimension-independent $\mathcal{O}(N^{-1/2})$ PoC rate under the $L^p(p\ge 2)$-Wasserstein metric. This is a different set of sufficient conditions compared to the existing results in the literature with the same rate, such as Lemma 5.1 in \cite{F.Delarue2019} and Theorem 4.2 in \cite{Szpruch2021}. Furthermore, note that the computational cost of simulating a particle system is $\mathcal{O}(N^2)$ at each time step, methods to reduce the cost can be found such as the projected particle method in \cite{DENIS2018} and the random batch method in \cite{JinShi2020JCP, JinShi2021SIAM}, but this issue is not considered in the present paper.

To demonstrate the eventual convergence of the numerical approximation and to give numerical tests, one needs to prove that the time-stepping scheme for the interacting particle system is convergent. Inspired by the numerical methods for SDEs (\cite{Higham2021,Kelly2023,MR3037286,MR3070913}), there are a few of results focus on the analysis of the time-stepping methods for MV-SDEs with various continuity coefficients. For example, for MV-SDEs with one-sided Lipschitz drift in the state variable, \cite{G.dosReis2022} analyzed strong convergence for the tamed EM method and the drift-implicit EM method, \cite{Reisinger2022} proposed an adaptive EM scheme, \cite{BaoJianhai2021} introduced a tamed Milstein scheme. While under a general Khasminskii-type condition, instead of imposing a one-sided and global Lipschitz condition on the drift and diffusion coefficients, respectively, \cite{Kumar2022} applied both tamed EM and Milstein schemes, \cite{Guo2024ESAIM}  developed the truncated EM scheme. Here, we use the tamed EM method for time discretization and show that it converges with order $\mathcal{O}(\Delta^{1/2})$ which is consistent with the classical results.

This paper is organized as follows: Section \ref{sec2} outlines notations and the preliminary analysis. Section \ref{sec3} shows that the strong PoC rate could be $\mathcal{O}(N^{-1/2})$ and dimension-independent for the proposed nonlinear MV-SDEs, and establishes the $L^p$-convergence of the tamed EM method simutaneously. Section \ref{Numerical simulation} presents numerical experiments to validate the theoretical PoC rate.

\section{Problem formulation}\label{sec2}
\subsection{Notations}\label{subsec2.1}
Throughout this paper, let $(\Omega ,\mathcal{F},\left\{\mathcal{F}_t\right\}_{t\ge 0},\mathbb{P})$ be a complete probability space with a filtration $\left\{\mathcal{F}_t\right\}_{t\ge 0}$ satisfying the usual conditions (i.e., it is right continuous and $\mathcal{F}_0$ contains all $\mathbb{P}$-null sets), and let $\mathbb{E}$ denote the expectation corresponding to $\mathbb{P}$. Let $\vert x\vert$ denote the Euclidean vector norm and $\langle x, y\rangle$ denote the inner product of vectors $x, y$. If $A$ is a vector or matrix, its transpose is denoted by $A^{\rm T}$. If A is a matrix, let $\vert A\vert$ denote its trace norm, i.e., $\vert A\vert = \sqrt{{\rm trace}(A^{\rm T} A)}$. For $a,b\in \mathbb R, a\vee b:=\max\{a,b\}, a\wedge b:=\min\{a,b\}$. 
In addition, we use $\mathcal{P}(\mathbb{R}^d)$ to denote the family of all probability measures on $(\mathbb{R}^d, \mathcal{B}(\mathbb{R}^d))$, where $\mathcal{B}(\mathbb{R}^d)$ denotes the Borel $\sigma$-field over $\mathbb{R}^d$, and define the subset of probability measures with finite $p$-th moment by 
\begin{align*}
\mathcal{P}_p(\mathbb{R}^d):=\left\{\mu\in\mathcal{P}(\mathbb{R}^d)\Big\vert \int_{\mathbb{R}^d}  \vert x\vert^p\mu(dx)<\infty\right\}.
\end{align*}
For $\mu, \sigma \in\mathcal{P}_p(\mathbb{R}^d), p\ge1$, the Wasserstein distance between $\mu$ and $\sigma$ is defined as
\begin{align*}
\mathbb{W}_p(\mu,\sigma):=\inf_{\pi\in\mathcal{C}(\mu,\sigma)}\left(\int_{\mathbb{R}^d\times\mathbb{R}^d}  \vert x-y\vert^p\pi(dx,dy)\right)^{1/p}
\end{align*}
where $\mathcal{C}(\mu,\sigma)$ means all the couplings of $\mu$ and $\sigma$, i.e., $\pi\in\mathcal{C}(\mu,\sigma)$ if and only if $\pi(\cdot,\mathbb{R}^d)=\mu(\cdot)$ and $\pi(\mathbb{R}^d,\cdot)=\sigma(\cdot)$. For $t\in[0,T]$, $p>0$, $L_t^{p}(\mathbb{R}^d)$ denotes the family of $\mathbb{R}^d$-valued $\mathcal{F}_t$-measurable random variables $\xi$ with $\mathbb{E}\vert \xi\vert^p<\infty$. As usual notations, “i.i.d.” means “independent and identically distributed”. 


\subsection{Problem formulation}\label{subsec2.2}
Let $x_0\in L_0^p(\mathbb{R}^d)$ for some given $p\ge 2$ and independent of the $m_0$-dimensional Brownian motion $W$. Let $T>0$, $a:\mathbb{R}^d\to \mathbb{R}^d$, in the following, we work on the nonlinear MV-SDEs: for $i=1,2,\dots, N$,
\begin{align}\label{NIPS}
{\rm d}x_t^i=& a(x_t^i){\rm d}t+f(x_t^i,\mu_t^{x^i}){\rm d}t+ g(x_t^i,\mu_t^{x^i}){\rm d}W^i_t,~t\in[0,T],
\end{align}
where $(x_0^i,W^i)_{1\le i\le N}$ are independent copies of $(x_0,W)$. For $t\in[0,T]$, define $\delta_t=[t/\delta]\delta$ for $\delta>0$ with $[t/\delta]$ denotes the integer part of $t/\delta$, $\hat\mu_t^{X,N,\delta}=\frac{1}{N}\sum_{j=1}^N\delta_{X_t^{j,N,\delta}}$. Let $X_0^{i,N,\delta}=x_0^i$ for $i=1,\dots, N$, $\tilde{a}:\mathbb{R}^d\to \mathbb{R}^d$, we also consider the interacting particle system
\begin{align}\label{IPS-2}
{\rm d}X_t^{i,N,\delta}=& \tilde{a}(X_{\delta_t}^{i,N,\delta}){\rm d}t+f(X_{\delta_t}^{i,N,\delta},\hat\mu_{\delta_t}^{X,N,\delta}){\rm d}t+g(X_{\delta_t}^{i,N,\delta},\hat\mu_{\delta_t}^{X,N,\delta}){\rm d}W^i_t,~t\in[0,T]
\end{align}
with coefficients taking the form of 
\begin{itemize}
\item[(i)] nonlinear distribution-dependence case:
\begin{align}\label{coefficient-1}
f(x,\mu)=A\left(\int_{\mathbb{R}^d}\kappa(x,y)\mu({\rm d}y) \right),~g(x,\mu)=B\left(\int_{\mathbb{R}^d}\zeta(x,y)\mu({\rm d}y) \right),
\end{align}
where $\kappa:\mathbb{R}^d\times\mathbb{R}^d\to \mathbb{R}^d$,  $A:\mathbb{R}^d\to\mathbb{R}^d$, $\zeta:\mathbb{R}^d\times\mathbb{R}^d\to \mathbb{R}^{d\times m_0}$, $B:\mathbb{R}^{d\times m_0}\to\mathbb{R}^{d\times m_0}$. 
\item[(ii)] more general nonlinear dependence case:
\begin{align}
f(x,\mu)=&A\left(\int_{\mathbb{R}^d}\kappa_1(x,y)\mu({\rm d}y),\cdots, \int_{\mathbb{R}^d}\kappa_q(x,y)\mu({\rm d}y) \right),\label{coefficient-2.1}\\
g(x,\mu)=&B\left(\int_{\mathbb{R}^d}\zeta_1(x,y)\mu({\rm d}y),\cdots, \int_{\mathbb{R}^d}\zeta_q(x,y)\mu({\rm d}y)\right)\label{coefficient-2.2}
\end{align}
where $q$ is a positive integer, $\kappa_j:\mathbb{R}^d\times\mathbb{R}^d\to \mathbb{R}^d, \zeta_j:\mathbb{R}^d\times\mathbb{R}^d\to \mathbb{R}^{d\times m_0}, j=1,\dots, q$, $A:\mathbb{R}^d\times\cdots\times\mathbb{R}^d\to\mathbb{R}^d$, $B:\mathbb{R}^{d\times m_0}\times\cdots\times\mathbb{R}^{d\times m_0}\to\mathbb{R}^{d\times m_0}$. 
\item[(iii)] higher-order interaction case:
\begin{align}
f(x,\mu)=&A\left(\int_{\mathbb{R}^d\times\cdots\times \mathbb{R}^d}\kappa(x,y_1,\dots, y_q)\mu({\rm d}y_1)\cdots\mu({\rm d}y_q) \right),\label{coefficient-3.1}\\
g(x,\mu)=&B\left(\int_{\mathbb{R}^d\times\cdots\times \mathbb{R}^d}\zeta(x,y_1,\dots, y_q) \mu({\rm d}y_1)\cdots\mu({\rm d}y_q) \right),\label{coefficient-3.2}
\end{align}
where $q$ is a positive integer, $\kappa:\mathbb{R}^d\times\cdots\times \mathbb{R}^d\to \mathbb{R}^d$, $\zeta:\mathbb{R}^d\times\cdots\times \mathbb{R}^d\to \mathbb{R}^{d\times m_0},$ $A:\mathbb{R}^d\to \mathbb{R}^d, B:\mathbb{R}^{d\times m_0}\to \mathbb{R}^{d\times m_0}$. 
\end{itemize}

Below, we make a comment on the three different forms of the coefficients.
\begin{remark}
Note that \eqref{coefficient-1} does not include the cases when several statistical quantities are involved in the coefficients, for example, if the coefficient depends on the variance, then both the first and second moments will be involved. These cases could be given by \eqref{coefficient-2.1} and \eqref{coefficient-2.2} by introducing multiple kernel functions. On the other hand, if the system depends on higher-order statistics, such as covariance, it is necessary to be described using equations with higher-order interactions \eqref{coefficient-3.1} and \eqref{coefficient-3.2}.

\end{remark}

The main aim in section \ref{sec3} is to establish the following error estimate:
\begin{align}\label{main aim}
\max_{1\le i\le N}\mathbb{E}\vert x_t^i-X_t^{i,N,\delta}\vert^p\le& C\left\{N^{-\frac{p}{2}}+(t-\delta_t)^{\frac{p}{2}}+\max_{1\le i\le N}\mathbb{E}\vert a(X_{t}^{i,N,\delta})-\tilde{a}(X_{t}^{i,N,\delta})\vert^p\right\}
\end{align}
for $p\ge 2$. Throughout this paper, let $C$ denote a generic positive constant that may change from line to line, and may depend on the initial data $x_0$, the terminal time $T$, and other fixed data (such as $p$ and the Lipschitz constants introduced below), but is always independent of the dimension $d$, the number of particles $N$ and the constant $\delta$. 

Next, let us explain why we focus on the systems \eqref{NIPS} and \eqref{IPS-2}, and consider the estimate \eqref{main aim}. This is primarily based on the following facts:
\begin{itemize}
\item In \eqref{IPS-2}, once we take $\tilde{a}=a, \delta=t$, then $\delta_t=t$ and \eqref{IPS-2} becomes
\begin{align}\label{IPS-3-1}
{\rm d}X_t^{i,N}=& a(X_{t}^{i,N}){\rm d}t+f(X_{t}^{i,N},\hat\mu_{t}^{X,N}){\rm d}t+g(X_{t}^{i,N},\hat\mu_{t}^{X,N}){\rm d}W^i_t,~X_0^{i,N}=x_0^i,
\end{align}
where we set $X_t^{i,N}:=X_t^{i,N,t},\hat\mu_{t}^{X,N}:=\hat\mu_{t}^{X,N,t}$ for the notational brevity. It is easy to see that \eqref{NIPS} and \eqref{IPS-3-1} are the non-interacting particle system and the interacting particle system corresponding to the MV-SDEs
\begin{align*}
{\rm d}x_t=& a(x_t){\rm d}t+f(x_t,\mu_t^{x}){\rm d}t+ g(x_t,\mu_t^{x}){\rm d}W_t,~t>0,
\end{align*}
with the initial data $x_0$, respectively, and as an immediate by-product of \eqref{main aim}, the PoC result 
\begin{align*}
\max_{1\le i\le N} \left\{\left(\mathbb{E}\vert x_t^i-X_t^{i,N}\vert^p\right)^{1/p}+\mathbb{E}\mathbb{W}_p(\mu_t^{x^i},\hat\mu_{t}^{X,N})\right\}\le CN^{-1/2}
\end{align*}
will be established.
\item If the coefficients of equation \eqref{NIPS} are all linearly growing, it is well known that the EM scheme is the simplest method to simulate the MV-SDEs (\cite{YunLi2023}): for a step size $\Delta>0$, $\underline{t}:=[t/\Delta]\Delta$, 
\begin{align*}
{\rm d}X_t^{i,N,\Delta}=& a(X_{\underline{t}}^{i,N,\Delta}){\rm d}t+f(X_{\underline{t}}^{i,N,\Delta},\hat\mu_{\underline{t}}^{X,N,\Delta}){\rm d}t+g(X_{\underline{t}}^{i,N,\Delta},\hat\mu_{\underline{t}}^{X,N,\Delta}){\rm d}W^i_t.
\end{align*}
Obviously, the scheme is included in \eqref{IPS-2} by taking $\tilde{a}=a$ and $\delta=\Delta$. However, this routine does not work for the MV-SDEs with non-globally Lipschitz continuous coefficients, and we must use the variants of the EM method. Inspired by \cite{Hutzenthaler2012, MaoXuerong2015, MaoXuerong2016}, tamed EM and truncated EM methods have been well developed for MV-SDEs (\cite{G.dosReis2022, Guo2024ESAIM,Kumar2022}). Suppose that $f$ and $g$ are Lipschitz continuous, and $a$ is one-sided Lipschitz continuous, if we set $\delta=\Delta$, take
\begin{align*}
\tilde{a}(x)=\frac{a(x)}{1+\Delta^{\gamma}\vert a(x)\vert}
\end{align*}
for $\gamma\in(0,1/2]$, or 
\begin{align*}
\tilde{a}(x)=a\left(\Big(\vert x\vert \wedge\varphi^{-1}(h(\Delta))\Big)\frac{x}{\vert x\vert}\right),
\end{align*}
where the precise definitions of the functions $\varphi$ and $h$ can be found in \cite[Section 3]{Guo2024ESAIM}, then \eqref{IPS-2} becomes the tamed EM and truncated EM schemes for system \eqref{NIPS}, respectively, and \eqref{main aim} gives the convergence result of the numerical scheme, as long as we further give the error estimate between $a$ and $\tilde{a}$.
\end{itemize}
Accordingly, we can simultaneously determine the convergence rate of both the propagation of chaos and the time discretization method, provided that \eqref{main aim} holds, rather than analyzing them separately.

\section{Main results}\label{sec3}
\subsection{Error estimate under the nonlinear distribution-dependence case}\label{subsec2.3}
The main objective of this paper is to establish the optimal convergence rate of the numerical method, particularly for the propagation of chaos, when the coefficients nonlinearly depend on the measure components. First, we consider the  equation \eqref{NIPS} with coefficients taking the form of \eqref{coefficient-1}. Before starting our work, the first natural step is to ensure that equations \eqref{NIPS} and \eqref{IPS-2} have unique solutions in this setting. We establish this under the following conditions.

\begin{assumption}\label{A2.1}
There are constants $L_1>0$ and $r\ge 1$ such that 
\begin{align*}
\langle x-y, a(x)-a(y)\rangle\le& L_1\vert x-y\vert^2,\\
\vert a(x)-a(y)\vert \le& L_1(1+\vert x\vert^r+\vert y\vert^r)\vert x-y\vert,\\
\vert a(0)\vert \le& L_1
\end{align*}
for all $x, y\in\mathbb{R}^d$.
\end{assumption}

\begin{assumption}\label{A2.2}
There are constants $L_2>0$ such that 
\begin{align*}
\vert A(x)-A(y)\vert\vee \vert B(x)-B(y)\vert\le& L_2\vert x-y\vert,\\
\vert\kappa(x,y)-\kappa(\bar{x},\bar{y})\vert\vee \vert \zeta(x,y)-\zeta(\bar{x},\bar{y})\vert\le& L_2(\vert x-\bar{x}\vert+\vert y-\bar{y}\vert),\\
\vert A(0)\vert \vee\vert \kappa(0,0)\vert\vee \vert B(0)\vert\vee\vert \zeta(0,0)\vert\le& L_2
\end{align*}
for all $x, y,\bar{x},\bar{y}\in\mathbb{R}^d$.
\end{assumption}

\begin{assumption}\label{A2.3}
Let $p\ge 2$, suppose that there exists a unique solution to \eqref{IPS-2} for any $i\in\{1,2,\dots, N\}$, and
\begin{align*}
\sup_{1\le i\le N}\left\{\mathbb{E}\left(\sup_{0\le t\le T}\vert X_t^{i,N,\delta}\vert^p\right)\vee
\mathbb{E}\left(\sup_{0\le t\le T}\vert \tilde{a}(X_t^{i,N,\delta})\vert^p\right)\right\}\le&  C,
\end{align*}
where $C$ is independent of $d$, $N$ and $\delta$.
\end{assumption}
Assumption \ref{A2.3} will be verified in the later when we take $\tilde{a}$ and $\delta$ as specific expressions.

\begin{remark} \label{remark1}\rm
It is well known that Assumptions \ref{A2.1} and \ref{A2.2} imply the following growth conditions of the coefficients:
\begin{align*}
\langle x, a(x)\rangle=&\langle x-0, a(x)-a(0)\rangle+\langle x,a(0)\rangle\le L_1\vert x\vert^2+\frac{\vert x\vert^2}{2}+\frac{\vert a(0)\vert^2}{2}\le  (L_1+\frac{1}{2})\vert x\vert^2+\frac{L_1^2}{2},\\
\vert a(x)\vert \le& \vert a(x)-a(0)\vert+ \vert a(0)\vert\le L_1(1+\vert x\vert+\vert x\vert^{r+1}),
\end{align*}
and similarly,
\begin{align*}
\vert A(x)\vert\vee\vert B(x)\vert\le&  L_2(1+\vert x\vert),~\vert \zeta(x,y)\vert\vee\vert \kappa(x,y)\vert\le L_2(1+\vert x\vert+\vert y\vert),
\end{align*}
\end{remark}

\begin{remark} \label{f to F}\rm
According to Assumption \ref{A2.2}, using the Kantorovich-Rubinstein dual representation of the Wasserstein distance $\mathbb{W}_1$:
\begin{align*}
\mathbb{W}_1(\mu,\sigma)=\sup\left\{\int_{\mathbb{R}^d} \phi(x) \mu({\rm d}x)-\int_{\mathbb{R}^d} \phi(x) \sigma({\rm d}x)\Bigg\vert \phi {\text ~Lipschitz~ continuous~with}~ \sup_{x\neq y}\frac{\vert \phi(x)-\phi(y)\vert }{\vert x-y\vert}\le 1\right\}
\end{align*}
and the fact that $\mathbb{W}_1(\mu,\sigma)\le \mathbb{W}_2(\mu,\sigma)$, it is easy to verify that the following inequalities subsequently hold
\begin{align*}
\vert f(x,\mu)-f(y,\sigma)\vert\le&C\big(\vert x-y\vert+\mathbb{W}_2(\mu,\sigma)\big),\\
\vert g(x,\mu)-g(y,\sigma)\vert\le&C\big(\vert x-y\vert+\mathbb{W}_2(\mu,\sigma)\big),\\
\vert f(0,\delta_0)\vert+\vert g(0,\delta_0)\vert\le& C
\end{align*}
for all $x,y\in\mathbb{R}^d$ and all $\mu,\sigma\in\mathcal{P}_2(\mathbb{R}^d)$, where $\delta_0$ denotes the delta distribution centered at the point 0. 
\end{remark}

 Based on Remark \ref{f to F}, we can obtain the existence, uniqueness and boundedness of the solution of equation \eqref{NIPS}  following Theorem 3.3 in \cite{G.dosReis2019}.
\begin{lemma}\label{x_t_bounded}
Let $p\ge 2$, for any $i=1,2,\dots,N$, under Assumptions \ref{A2.1} and \ref{A2.2}, there exists a unique solution $x_t^i$ to the equation \eqref{NIPS}, and the solution satisfies
\begin{align*}
\mathbb{E}\left(\sup_{0\le t\le T}\vert x_t^i\vert^p\right)\le C.
\end{align*} 
\end{lemma}

\begin{theorem}\label{theorem_x^i-x^i,K,N}
Let $p\ge 2$, it holds under Assumptions \ref{A2.1}-\ref{A2.3} that
\begin{align*}
\max_{1\le i\le N}\mathbb{E}\left(\sup_{0\le t\le T}\vert x_t^i-X_t^{i,N,\delta}\vert^p\right)\le& C\left\{N^{-\frac{p}{2}}+\sup_{0\le t\le T}(t-\delta_t)^{\frac{p}{2}}+\max_{1\le i\le N}\mathbb{E}\left(\sup_{0\le t\le T}\vert a(X_{t}^{i,N,\delta})-\tilde{a}(X_{t}^{i,N,\delta})\vert^p\right)\right\}.
\end{align*}
\end{theorem}
Before giving the proof of this theorem, let us introduce some definitions for simplicity and give some lemmas. Define
\begin{align*}
\hat{\kappa}(x_t^i):=\int_{\mathbb{R}^d}\kappa(x_t^i,y)\mu_t^{x^i}({\rm d}y), \quad\hat{\zeta}(x_t^i):=\int_{\mathbb{R}^d}\zeta(x_t^i,y)\mu_t^{x^i}({\rm d}y)
\end{align*}
for notation brevity, then equation \eqref{NIPS} with coefficients given by \eqref{coefficient-1} equals
\begin{align}\label{NIPS-2.3}
x_t^i=x_0^i+ \int_0^t a(x_s^i){\rm d}s+ \int_0^t A\left(\hat{\kappa}(x_s^i)\right){\rm d}s+\int_0^t B\left(\hat{\zeta}(x_s^i)\right){\rm d}W_s^i.
\end{align}
For $i=1,2,\dots, N$ and $x^i\in\mathbb{R}^d$, define
\begin{align*}
\hat{\kappa}^{N}(x^i):=\frac{1}{N}\sum_{j=1}^N\kappa(x^i,x^j),\quad \hat{\zeta}^{N}(x^i):=\frac{1}{N}\sum_{j=1}^N\zeta(x^i,x^j),
\end{align*}
then \eqref{IPS-2} with coefficients given by \eqref{coefficient-1} can be rewritten as
 \begin{align}\label{PPS-2}
X_t^{i,N,\delta}=x_0^i+ \int_0^t \tilde{a}(X_{\delta_s}^{i,N,\delta}){\rm d}s+\int_0^t A\left(\hat{\kappa}^{N}(X_{\delta_s}^{i,N,\delta})\right){\rm d}s+\int_0^t B\left(\hat{\zeta}^{N}(X_{\delta_s}^{i,N,\delta})\right){\rm d}W_s^i.
\end{align}

\begin{lemma}[Conditional Rosenthal-type inequality (Theorem 2.2 in \cite{DemeiYuan2015})]\label{conditional Rosenthal-type inequality}
Let $\mathcal{G}$ be a sub-$\sigma$-algebra of $\mathcal{F}$, and let $A_1, A_2,\dots, A_n$ be $\mathcal{G}$-independent random variables with $\mathbb{E}[A_i\vert\mathcal{G}]=0$ a.s. for all $i$. For $p>2$, there exists a constant $C_p$ depending only on $p$ such that
\begin{align*}
\mathbb{E}\left[\left\vert\sum_{i=1}^n A_i\right\vert^p\Big\vert\mathcal{G}\right]\le C_p\max\left\{\sum_{i=1}^n\mathbb{E}\Big[\left\vert A_i\right\vert^p\big\vert\mathcal{G}\Big],\left(\sum_{i=1}^n\mathbb{E}\Big[\left\vert A_i\right\vert^2\big\vert\mathcal{G}\Big] \right)^{\frac{p}{2}}\right\}.
\end{align*}
\end{lemma}

\begin{lemma}\label{kappa-kappa^N}
Let $p\ge2$, for any $t\in[0,T]$, it holds under Assumptions \ref{A2.1} and \ref{A2.2} that
\begin{align*}
\mathbb{E}\left\vert \hat{\kappa}(x_t^i)-\hat{\kappa}^{N}(x_t^i)\right\vert^p\vee\mathbb{E}\left\vert \hat{\zeta}(x_t^i)-\hat{\zeta}^{N}(x_t^i)\right\vert^p\le CN^{-\frac{p}{2}}.
\end{align*}
\end{lemma}
\begin{proof}
According to Jensen's inequality, we have 
\begin{align*}
&\mathbb{E}\vert \hat{\kappa}(x_t^i)-\hat{\kappa}^{N}(x_t^i)\vert^p=\mathbb{E}\left\vert \frac{1}{N}\sum_{j=1}^N\int_{\mathbb{R}^d}\left(\kappa(x_t^i,y)-\kappa(x_t^i,x_t^j)\right)\mu_t^{x^i}({\rm d}y)\right\vert^p\notag\\
\le&\frac{2^{p-1}}{N^{p}}\left(\mathbb{E}\left\vert \int_{\mathbb{R}^d}\left(\kappa(x_t^i,y)-\kappa(x_t^i,x_t^i)\right)\mu_t^{x^i}({\rm d}y)\right\vert^p+\mathbb{E}\left\vert\sum_{j\neq i}\int_{\mathbb{R}^d}\left(\kappa(x_t^i,y)-\kappa(x_t^i,x_t^j)\right)\mu_t^{x^i}({\rm d}y)\right\vert^p\right)
\end{align*}
For $p>2$, using the tower property of the expectation, Lemma \ref{conditional Rosenthal-type inequality}, and the H\"older inequality for the conditional expectation, we can deduce that
\begin{align*}
&\mathbb{E}\left\vert\sum_{j\neq i}\int_{\mathbb{R}^d}\left(\kappa(x_t^i,y)-\kappa(x_t^i,x_t^j)\right)\mu_t^{x^i}({\rm d}y)\right\vert^p\notag\\
=&\mathbb{E}\left\{\mathbb{E}\left[\left\vert\sum_{j\neq i}\int_{\mathbb{R}^d}\left(\kappa(x_t^i,y)-\kappa(x_t^i,x_t^j)\right)\mu_t^{x^i}({\rm d}y)\right\vert^p\Big\vert \sigma(x_t^i)\right]\right\}\notag\\
\le&\mathbb{E}\left\{\sum_{j\neq i}\mathbb{E}\left[\left\vert\int_{\mathbb{R}^d}\left(\kappa(x_t^i,y)-\kappa(x_t^i,x_t^j)\right)\mu_t^{x^i}({\rm d}y)\right\vert^p\Big\vert \sigma(x_t^i)\right]\right\}\notag\\
&+\mathbb{E}\left\{\left(\sum_{j\neq i}\mathbb{E}\left[\left\vert\int_{\mathbb{R}^d}\left(\kappa(x_t^i,y)-\kappa(x_t^i,x_t^j)\right)\mu_t^{x^i}({\rm d}y)\right\vert^2\Big\vert \sigma(x_t^i)\right]\right)^{\frac{p}{2}}\right\}\notag\\
\le&\sum_{j\neq i}\mathbb{E}\left\vert\int_{\mathbb{R}^d}\left(\kappa(x_t^i,y)-\kappa(x_t^i,x_t^j)\right)\mu_t^{x^i}({\rm d}y)\right\vert^p+(N-1)^{\frac{p}{2}-1}\sum_{j\neq i}\mathbb{E}\left\vert\int_{\mathbb{R}^d}\left(\kappa(x_t^i,y)-\kappa(x_t^i,x_t^j)\right)\mu_t^{x^i}({\rm d}y)\right\vert^p,
\end{align*}
in which $\sigma(x_t^i)$ denotes the $\sigma$-algebra generated by $x_t^i$. Then, applying the H\"older inequality, Assumption \ref{A2.2} and Lemma \ref{x_t_bounded}, yields that
\begin{align}\label{F_1-F_1^N}
&\mathbb{E}\vert \hat{\kappa}(x_t^i)-\hat{\kappa}^{N}(x_t^i)\vert^p\notag\\
\le&CN^{-p}\mathbb{E}\left\vert \int_{\mathbb{R}^d}\left(\kappa(x_t^i,y)-\kappa(x_t^i,x_t^i)\right)\mu_t^{x^i}({\rm d}y)\right\vert^p+CN^{-p}\sum_{j\neq i}\mathbb{E}\left\vert\int_{\mathbb{R}^d}\left(\kappa(x_t^i,y)-\kappa(x_t^i,x_t^j)\right)\mu_t^{x^i}({\rm d}y)\right\vert^p\notag\\
&+CN^{-\frac{p}{2}-1}\sum_{j\neq i}\mathbb{E}\left\vert\int_{\mathbb{R}^d}\left(\kappa(x_t^i,y)-\kappa(x_t^i,x_t^j)\right)\mu_t^{x^i}({\rm d}y)\right\vert^p\notag\\
\le&C(N^{-p}+N^{-\frac{p}{2}-1})\sum_{j=1}^N\mathbb{E}\left\vert\int_{\mathbb{R}^d}\left(\kappa(x_t^i,y)-\kappa(x_t^i,x_t^j)\right)\mu_t^{x^i}({\rm d}y)\right\vert^p\notag\\
\le&CN^{-\frac{p}{2}}(N^{-\frac{p}{2}}+N^{-1})\sum_{j=1}^N\mathbb{E}\int_{\mathbb{R}^d}\left\vert\kappa(x_t^i,y)-\kappa(x_t^i,x_t^j)\right\vert^p\mu_t^{x^i}({\rm d}y)\notag\\
\le&CN^{-\frac{p}{2}}(N^{-\frac{p}{2}}+N^{-1})\sum_{j=1}^N\mathbb{E}\int_{\mathbb{R}^d}\vert y-x_t^j\vert^p\mu_t^{x^i}({\rm d}y)\notag\\
\le&CN^{-\frac{p}{2}}(N^{-\frac{p}{2}}+N^{-1})\sum_{j=1}^N\mathbb{E}\left(\mathbb{E}\vert x_t^i\vert^p+\vert x_t^j\vert^p\right)\notag\\
\le&CN^{-\frac{p}{2}}(N^{1-\frac{p}{2}}+1)\notag\\
\le&CN^{-\frac{p}{2}}.
\end{align}
Moreover, following the H\"older inequality once again, one has
\begin{align*}
\mathbb{E}\vert \hat{\kappa}(x_t^i)-\hat{\kappa}^{N}(x_t^i)\vert^2\le \left(\mathbb{E}\vert \hat{\kappa}(x_t^i)-\hat{\kappa}^{N}(x_t^i)\vert^p\right)^{\frac{2}{p}}\le CN^{-1}.
\end{align*}
Hence
\begin{align*}
\mathbb{E}\vert \hat{\kappa}(x_t^i)-\hat{\kappa}^{N}(x_t^i)\vert^p\le CN^{-\frac{p}{2}}
\end{align*}
holds for all $p\ge 2$. Analogously we can also get
\begin{align*}
\mathbb{E}\left\vert \hat{\zeta}(x_t^i)-\hat{\zeta}^{N}(x_t^i)\right\vert^p\le CN^{-\frac{p}{2}}
\end{align*}
for $p\ge 2$ from Assumption \ref{A2.2}.
\end{proof}


\begin{lemma}\label{X_t-X_delta_t}
Let $p\ge 2$, for any $t\in[0,T]$, it holds under Assumptions \ref{A2.1}, \ref{A2.2} and \ref{A2.3} that
\begin{align*}
\mathbb{E} \vert X_t^{i,N,\delta}-X_{\delta_t}^{i,N,\delta}\vert^p\le C(t-\delta_t)^{\frac{p}{2}}.
\end{align*}
\end{lemma}
\begin{proof}
Since $\delta_t=[t/\delta]\delta, \delta>0$, it is easy to know that $0\le \delta_t\le t$, then
 \begin{align*}
X_t^{i,N,\delta}=&X_{\delta_t}^{i,N,\delta}+\int_{\delta_t}^t \tilde{a}(X_{\delta_s}^{i,N,\delta}){\rm d}s+\int_{\delta_t}^tA\left(\hat{\kappa}^{N}(X_{\delta_s}^{i,N,\delta})\right){\rm d}s+\int_{\delta_t}^tB\left(\hat{\zeta}^{N}(X_{\delta_s}^{i,N,\delta})\right){\rm d}W_s^i.
\end{align*}
On the other hand, for $s\in[\delta_t,t]$, since 
 \begin{align*}
[\delta_t/\delta]\delta\le [s/\delta]\delta\le [t/\delta]\delta, 
\end{align*}
and 
 \begin{align*}
[\delta_t/\delta]\delta=\left[\frac{[t/\delta]\delta}{\delta}\right]\delta=[t/\delta]\delta,
\end{align*}
we can get that $\delta_s=\delta_t$. Hence, using Jensen's inequality, the B-D-G inequality, and Remark \ref{remark1}, we have
 \begin{align*}
&\mathbb{E}\vert X_t^{i,N,\delta}-X_{\delta_t}^{i,N,\delta}\vert^p\\
= &\mathbb{E}\left\vert \int_{\delta_t}^t \tilde{a}(X_{\delta_s}^{i,N,\delta}){\rm d}s+\int_{\delta_t}^tA\left(\hat{\kappa}^{N}(X_{\delta_s}^{i,N,\delta})\right){\rm d}s+\int_{\delta_t}^tB\left(\hat{\zeta}^{N}(X_{\delta_s}^{i,N,\delta})\right){\rm d}W_s^i\right\vert^p\\
\le &C(t-\delta_t)^p\mathbb{E}\left\vert \tilde{a}(X_{\delta_t}^{i,N,\delta})\right\vert^p+C(t-\delta_t)^p\mathbb{E}\left\vert A\left(\hat{\kappa}^{N}(X_{\delta_t}^{i,N,\delta})\right)\right\vert^p+C(t-\delta_t)^{\frac{p}{2}}\mathbb{E}\left\vert B\left(\hat{\zeta}^{N}(X_{\delta_t}^{i,N,\delta})\right)\right\vert^p\\
\le &C(t-\delta_t)^p\mathbb{E}\left\vert \tilde{a}(X_{\delta_t}^{i,N,\delta})\right\vert^p+C(t-\delta_t)^p\left(1+\mathbb{E}\left\vert \hat{\kappa}^{N}(X_{\delta_t}^{i,N,\delta})\right\vert^p\right)+C(t-\delta_t)^{\frac{p}{2}}\left(1+\mathbb{E}\left\vert \hat{\zeta}^{N}(X_{\delta_s}^{i,N,\delta})\right\vert^p\right)\\
\le &C(t-\delta_t)^p\mathbb{E}\left\vert \tilde{a}(X_{\delta_t}^{i,N,\delta})\right\vert^p+C(t-\delta_t)^p\left(1+\frac{1}{N}\sum_{j=1}^N\mathbb{E}\left\vert\kappa(X_{\delta_s}^{i,N,\delta},X_{\delta_s}^{j,N,\delta})\right\vert^p\right) \\
&+C(t-\delta_t)^{\frac{p}{2}}\left(1+\frac{1}{N}\sum_{j=1}^N\mathbb{E}\left\vert\zeta(X_{\delta_s}^{i,N,\delta},X_{\delta_s}^{j,N,\delta})\right\vert^p\right) \\
\le &C(t-\delta_t)^p\mathbb{E}\left\vert \tilde{a}(X_{\delta_t}^{i,N,\delta})\right\vert^p+C\left((t-\delta_t)^p+(t-\delta_t)^{\frac{p}{2}}\right)\left(1+\mathbb{E}\vert X_{\delta_t}^{i,N,\delta}\vert^p+\frac{1}{N}\sum_{j=1}^N\mathbb{E}\vert X_{\delta_t}^{j,N,\delta}\vert^p\right)\\
\le &C(t-\delta_t)^{\frac{p}{2}}\left((t-\delta_t)^{\frac{p}{2}}\mathbb{E}\left\vert \tilde{a}(X_{\delta_t}^{i,N,\delta})\right\vert^p+\left(1+(t-\delta_t)^{\frac{p}{2}}\right)\left(1+\mathbb{E}\vert X_{\delta_t}^{i,N,\delta}\vert^p\right)\right).
\end{align*}
Since $t-\delta_t\le t\le T$, applying Assumption \ref{A2.3}, it holds that
 \begin{align*}
\mathbb{E}\vert X_t^{i,N,\delta}-X_{\delta_t}^{i,N,\delta}\vert^p\le C(t-\delta_t)^{\frac{p}{2}}.
\end{align*}
\end{proof}

\begin{proof}[Proof of Theorem \ref{theorem_x^i-x^i,K,N}]
For any $i=1,2,\dots,N$, according to \eqref{NIPS-2.3} and \eqref{PPS-2}, we have
\begin{align*}
e_t^{i,N}:=x_t^i-X_t^{i,N,\delta}&=\int_0^t\left(a(x_s^i)-\tilde{a}(X_{\delta_s}^{i,N,\delta})\right){\rm d}s+\int_0^t\left( A\left(\hat{\kappa}(x_s^i)\right)-A\left(\hat{\kappa}^{N}(X_{\delta_s}^{i,N,\delta})\right)\right){\rm d}s\notag\\
&\quad+\int_0^t \left(B\left(\hat{\zeta}(x_s^i)\right)-B\left(\hat{\zeta}^{N}(X_{\delta_s}^{i,N,\delta})\right)\right){\rm d}W_s^i.
\end{align*}
For $p>2$, applying the It\^o formula, the Cauchy-Schwarz inequality, and Assumption \ref{A2.2}, one has
\begin{align}\label{e^p}
\vert e_t^{i,N}\vert^p\le & p\int_0^t \vert e_s^{i,N}\vert^{p-2}\left\langle e_s^{i,N},a(x_s^i)-\tilde{a}(X_{\delta_s}^{i,N,\delta})\right\rangle{\rm d}s+ pL_2\int_0^t \vert e_s^{i,N}\vert^{p-1}\left\vert \hat{\kappa}(x_s^i)-\hat{\kappa}^{N}(X_{\delta_s}^{i,N,\delta})\right\vert{\rm d}s\notag\\
&+\frac{p(p-1)L_2^2}{2}\int_0^t \vert e_s^{i,N}\vert^{p-2}\left\vert \hat{\zeta}(x_s^i)-\hat{\zeta}^{N}(X_{\delta_s}^{i,N,\delta})\right\vert^2{\rm d}s\notag\\
&+p\int_0^t\vert e_s^{i,N}\vert^{p-2}\left\langle e_s^{i,N},B\left(\hat{\zeta}(x_s^i)\right)-B\left(\hat{\zeta}^{N}(X_{\delta_s}^{i,N,\delta})\right)\right\rangle{\rm d}W_s^i\notag\\
=:& J_1+J_2+J_3+J_4.
\end{align}
Following Assumption \ref{A2.1} and the Young inequality, it is easy to see that
\begin{align}\label{J_1}
J_1\le& p\int_0^t \vert e_s^{i,N}\vert^{p-2}\left\langle e_s^{i,N},a(x_s^i)-a(X_s^{i,N,\delta})\right\rangle{\rm d}s+p\int_0^t \vert e_s^{i,N}\vert^{p-1}\vert a(X_s^{i,N,\delta})-a(X_{\delta_s}^{i,N,\delta})\vert{\rm d}s\notag\\
&+p\int_0^t \vert e_s^{i,N}\vert^{p-1}\vert a(X_{\delta_s}^{i,N,\delta})-\tilde{a}(X_{\delta_s}^{i,N,\delta})\vert{\rm d}s\notag\\
\le& C\int_0^t \vert e_s^{i,N}\vert^p{\rm d}s+C\int_0^t (1+\vert X_s^{i,N,\delta}\vert^r+\vert X_{\delta_s}^{i,N,\delta}\vert^r)^p\vert X_s^{i,N,\delta}-X_{\delta_s}^{i,N,\delta}\vert^p{\rm d}s\notag\\
&+\int_0^t \vert a(X_{\delta_s}^{i,N,\delta})-\tilde{a}(X_{\delta_s}^{i,N,\delta})\vert^p{\rm d}s.
\end{align}
Similarly, we can also deduce from the Young inequality and Assumption \ref{A2.1} that
\begin{align}\label{J_2}
J_2\le&pL_2\int_0^t \vert e_s^{i,N}\vert^{p-1}\left(\vert\hat{\kappa}(x_s^i)-\hat{\kappa}^{N}(x_s^i)\vert+\vert\hat{\kappa}^{N}(x_s^i)-\hat{\kappa}^{N}(X_s^{i,N,\delta})\vert+\vert\hat{\kappa}^{N}(X_s^{i,N,\delta})-\hat{\kappa}^{N}(X_{\delta_s}^{i,N,\delta})\vert\right){\rm d}s\notag\\
\le&C\int_0^t\vert e_s^{i,N}\vert^p{\rm d}s+C\int_0^t\vert \hat{\kappa}(x_s^i)-\hat{\kappa}^{N}(x_s^i)\vert^p{\rm d}s+\frac{C}{N}\sum_{j=1}^N\int_0^t\left(\vert e_s^{i,N}\vert^p+\vert e_s^{j,N}\vert^p\right){\rm d}s\notag\\
&+\frac{C}{N}\sum_{j=1}^N\int_0^t\left(\vert X_s^{i,N,\delta}-X_{\delta_s}^{i,N,\delta}\vert^p+\vert X_s^{j,N,\delta}-X_{\delta_s}^{j,N,\delta}\vert^p\right){\rm d}s,
\end{align}
and
\begin{align}\label{J_3}
J_3\le&C\int_0^t \vert e_s^{i,N}\vert^p{\rm d}s+C\int_0^t \left\vert \hat{\zeta}(x_s^i)-\hat{\zeta}^{N}(x_s^i)\right\vert^p{\rm d}s+\frac{C}{N}\sum_{j=1}^N\int_0^t\left(\vert e_s^{i,N}\vert^p+\vert e_s^{j,N}\vert^p\right){\rm d}s\notag\\
&+\frac{C}{N}\sum_{j=1}^N\int_0^t\left(\vert X_s^{i,N,\delta}-X_{\delta_s}^{i,N,\delta}\vert^p+\vert X_s^{j,N,\delta}-X_{\delta_s}^{j,N,\delta}\vert^p\right){\rm d}s.
\end{align}
Substituting \eqref{J_1}-\eqref{J_3} into \eqref{e^p}, it gives that
\begin{align*}
\vert e_t^{i,N}\vert^p\le&C\int_0^t \vert e_s^{i,N}\vert^p{\rm d}s+C\int_0^t (1+\vert X_s^{i,N,\delta}\vert^r+\vert X_{\delta_s}^{i,N,\delta}\vert^r)^p\vert X_s^{i,N,\delta}-X_{\delta_s}^{i,N,\delta}\vert^p{\rm d}s\notag\\
&+\int_0^t \vert a(X_{\delta_s}^{i,N,\delta})-\tilde{a}(X_{\delta_s}^{i,N,\delta})\vert^p{\rm d}s+C\int_0^t\vert \hat{\kappa}(x_s^i)-\hat{\kappa}^{N}(x_s^i)\vert^p{\rm d}s+\frac{C}{N}\sum_{j=1}^N\int_0^t\vert e_s^{j,N}\vert^p{\rm d}s\notag\\
&+\frac{C}{N}\sum_{j=1}^N\int_0^t\left(\vert X_s^{i,N,\delta}-X_{\delta_s}^{i,N,\delta}\vert^p+\vert X_s^{j,N,\delta}-X_{\delta_s}^{j,N,\delta}\vert^p\right){\rm d}s+C\int_0^t \left\vert \hat{\zeta}(x_s^i)-\hat{\zeta}^{N}(x_s^i)\right\vert^p{\rm d}s
+J_4.
\end{align*}
Then for any $T_1\in[0,T]$, since the $p$th moments of $x_t^i$ and $X_t^{i,N,\delta}$ are all bounded, we have
\begin{align}\label{Ee^p}
\mathbb{E}\left(\sup_{0\le t\le T_1}\vert e_t^{i,N}\vert^p\right)\le &C\int_0^{T_1}\mathbb{E}\vert e_t^{i,N}\vert^p{\rm d}t+C\int_0^{T_1}\left(\mathbb{E} (1+\vert X_t^{i,N,\delta}\vert^r+\vert X_{\delta_t}^{i,N,\delta}\vert^r)^{2p}\mathbb{E}\vert X_t^{i,N,\delta}-X_{\delta_t}^{i,N,\delta}\vert^{2p}\right)^{1/2}{\rm d}t\notag\\
&+\int_0^{T_1}\mathbb{E} \vert a(X_{\delta_t}^{i,N,\delta})-\tilde{a}(X_{\delta_t}^{i,N,\delta})\vert^p{\rm d}t+C\int_0^{T_1}\mathbb{E} \vert \hat{\kappa}(x_t^i)-\hat{\kappa}^{N}(x_t^i)\vert^p{\rm d}t\notag\\
&+C\int_0^{T_1}\mathbb{E}\vert X_t^{i,N,\delta}-X_{\delta_t}^{i,N,\delta}\vert^p{\rm d}t+C\int_0^{T_1}\mathbb{E} \left\vert \hat{\zeta}(x_t^i)-\hat{\zeta}^{N}(x_t^i)\right\vert^p{\rm d}t\notag\\
&+C\mathbb{E}\sup_{0\le t\le T_1}\left(\int_0^t\vert e_s^{i,N}\vert^{p-2}\left\langle e_s^{i,N},B\left(\hat{\zeta}(x_s^i)\right)-B\left(\hat{\zeta}^{N}(X_{\delta_s}^{i,N,\delta})\right)\right\rangle{\rm d}W_s^i\right).
\end{align}
Using the B-D-G inequality, similar to \eqref{J_3}, it holds that
\begin{align}\label{J_4}
&\mathbb{E}\sup_{0\le t\le T_1}\left(\int_0^t\vert e_s^{i,N}\vert^{p-1}\left\vert\hat{\zeta}(x_s^i)-\hat{\zeta}^{N}(X_{\delta_s}^{i,N,\delta})\right\vert{\rm d}W_s^i\right)\notag\\
\le& C\mathbb{E}\left(\int_0^{T_1}\vert e_t^{i,N}\vert^{2p-2}\vert \hat{\zeta}(x_t^i)-\hat{\zeta}^{N}(X_{\delta_t}^{i,N,\delta})\vert^2{\rm d}t\right)^{1/2}\notag\\
\le&C\mathbb{E}\left(\sup_{0\le t\le T_1}\vert e_t^{i,N}\vert^{p}\int_0^{T_1}\vert e_t^{i,N}\vert^{p-2}\vert \hat{\zeta}(x_t^i)-\hat{\zeta}^{N}(X_{\delta_t}^{i,N,\delta})\vert^{2}{\rm d}t\right)^{1/2}\notag\\
\le&\frac{1}{2}\mathbb{E}\left(\sup_{0\le t\le T_1}\vert e_t^{i,N}\vert^p\right)+C\int_0^{T_1} \mathbb{E}\vert e_t^{i,N}\vert^p{\rm d}t+C\int_0^{T_1} \mathbb{E}\left\vert \hat{\zeta}(x_t^i)-\hat{\zeta}^{N}(x_t^i)\right\vert^p{\rm d}t+C\int_0^{T_1}\mathbb{E}\vert X_t^{i,N,\delta}-X_{\delta_t}^{i,N,\delta}\vert^p{\rm d}t.
\end{align}
Thus, combining \eqref{Ee^p} and \eqref{J_4}, one has
\begin{align}\label{Esupe^p}
\mathbb{E}\left(\sup_{0\le t\le T_1}\vert e_t^{i,N}\vert^p\right)\le &C\int_0^{T_1}\mathbb{E}\vert e_t^{i,N}\vert^p{\rm d}t+C\int_0^{T_1}\left(\mathbb{E} (1+\vert X_t^{i,N,\delta}\vert^r+\vert X_{\delta_t}^{i,N,\delta}\vert^r)^{2p}\mathbb{E}\vert X_t^{i,N,\delta}-X_{\delta_t}^{i,N,\delta}\vert^{2p}\right)^{1/2}{\rm d}t\notag\\
&+\int_0^{T_1}\mathbb{E} \vert a(X_{\delta_t}^{i,N,\delta})-\tilde{a}(X_{\delta_t}^{i,N,\delta})\vert^p{\rm d}t+C\int_0^{T_1}\mathbb{E} \vert \hat{\kappa}(x_t^i)-\hat{\kappa}^{N}(x_t^i)\vert^p{\rm d}t\notag\\
&+C\int_0^{T_1}\mathbb{E}\vert X_t^{i,N,\delta}-X_{\delta_t}^{i,N,\delta}\vert^p{\rm d}t+C\int_0^{T_1}\mathbb{E} \left\vert \hat{\zeta}(x_t^i)-\hat{\zeta}^{N}(x_t^i)\right\vert^p{\rm d}t.
\end{align}
Substituting Lemmas \ref{kappa-kappa^N} and \ref{X_t-X_delta_t} into \eqref{Esupe^p}, using Assumption \ref{A2.3} again, it gives
\begin{align*}
\mathbb{E}\left(\sup_{0\le t\le T_1}\vert e_t^{i,N}\vert^p\right)\le& C\int_0^{T_1}\mathbb{E}\vert e_t^{i,N}\vert^p{\rm d}t+C\int_0^{T_1}(t-\delta_t)^{\frac{p}{2}}\left(1+\left\{\mathbb{E}\left(\sup_{0\le s\le t}\vert X_s^{i,N,\delta}\vert^{2pr}\right)\right\}^{1/2}\right){\rm d}t\\
&+\int_0^{T_1}\mathbb{E} \vert a(X_{\delta_t}^{i,N,\delta})-\tilde{a}(X_{\delta_t}^{i,N,\delta})\vert^p{\rm d}t+CN^{-\frac{p}{2}}\\
\le &C\int_0^{T_1}\mathbb{E}\left(\sup_{0\le s\le t}\vert e_s^{i,N}\vert^p\right){\rm d}t+CN^{-\frac{p}{2}}+C\int_0^{T_1}(t-\delta_t)^{\frac{p}{2}}{\rm d}t\\
&+C\int_0^{T_1}\mathbb{E} \vert a(X_{\delta_t}^{i,N,\delta})-\tilde{a}(X_{\delta_t}^{i,N,\delta})\vert^p{\rm d}t.
\end{align*}
Therefore, it follows from the Gronwall inequality that
\begin{align*}
\mathbb{E}\left(\sup_{0\le t\le T_1}\vert e_t^{i,N}\vert^p\right)\le& CN^{-\frac{p}{2}}+C\sup_{0\le t\le T_1}(t-\delta_t)^{\frac{p}{2}}+C\mathbb{E}\left(\sup_{0\le t\le T_1}\vert a(X_{t}^{i,N,\delta})-\tilde{a}(X_{t}^{i,N,\delta})\vert^p\right).
\end{align*}
Then the assertion follows. For $p=2$, the result can be obtained in the same way.
\end{proof}

\subsection{Error estimate under the more general nonlinear dependence case}\label{subsec2.3}
In the following two subsections, we discuss more general cases.  In this subsection, consider equation \eqref{NIPS} with coefficients taking the form of \eqref{coefficient-2.1} and \eqref{coefficient-2.2}. Suppose that $a$ still satisfies Assumption \ref{A2.2}, and the other terms all satisfy the global Lipschitz conditions, the error estimate \eqref{main aim} for such systems can be similarly performed. There is no big difference for the error estimate, while the details could be tedious.

\begin{assumption}\label{A2.4}
There are constants $L_3>0$ such that 
\begin{align*}
\vert A(x_1,\dots, x_q)-A(y_1, \dots, y_q)\vert\vee   \vert B(x_1,\dots, x_q)-B(y_1, \dots, y_q)\vert\le& L_3\sum_{m=1}^q\vert x_j-y_j\vert,\\
\vert A(0,\dots, 0)\vert\vee \vert B(0,\dots, 0)\vert \le&  L_3
\end{align*}
for all $x_j, y_j\in\mathbb{R}^d, j=1,2,\dots, q$. Moreover, for any $j=1,2,\dots, q$, there is a constant $K_j>0$ such that
\begin{align*}
\vert \kappa_j(x,y)-\kappa_j(\bar{x},\bar{y})\vert\vee   \vert \zeta_j(x,y)-\zeta_j(\bar{x},\bar{y})\vert\le&  K_j(\vert x-\bar{x}\vert+\vert y-\bar{y}\vert),\\
\vert \kappa_j(0,0)\vert\vee\vert \zeta_j(0,0)\vert\le&  K_j
\end{align*}
for all $x, y,\bar{x},\bar{y}\in\mathbb{R}^d$.
\end{assumption}

\begin{remark}\label{remark2.3}
Under Assumption \ref{A2.4}, we can also verify that the following inequalities hold
\begin{align*}
\vert f(x,\mu)-f(y,\sigma)\vert\le&C\big(\vert x-y\vert+\mathbb{W}_2(\mu,\sigma)\big),\\
\vert g(x,\mu)-g(y,\sigma)\vert\le&C\big(\vert x-y\vert+\mathbb{W}_2(\mu,\sigma)\big),\\
\vert f(0,\delta_0)\vert+\vert g(0,\delta_0)\vert\le& C
\end{align*}
for all $x,y\in\mathbb{R}^d$ and all $\mu,\sigma\in\mathcal{P}_2(\mathbb{R}^d)$, which implies the existence, uniqueness and boundedness of the solution of MV-SDEs \eqref{NIPS}.
\end{remark}

\begin{theorem}\label{theorem_x^i-x^i,K,N-2}
Let $p\ge 2$, it holds under Assumptions \ref{A2.1}, \ref{A2.3}, and \ref{A2.4} that
\begin{align*}
\max_{1\le i\le N}\mathbb{E}\left(\sup_{0\le t\le T}\vert x_t^i-X_t^{i,N,\delta}\vert^p\right)\le& CN\left\{^{-\frac{p}{2}}+\sup_{0\le t\le T}(t-\delta_t)^{\frac{p}{2}}+\max_{1\le i\le N}\mathbb{E}\left(\sup_{0\le t\le T}\vert a(X_{t}^{i,N,\delta})-\tilde{a}(X_{t}^{i,N,\delta})\vert^p\right)\right\}.
\end{align*}
\end{theorem}
The proof is similar to that of Theorem \ref{theorem_x^i-x^i,K,N}, so we omit it.

\subsection{Error estimate under higher-order interaction case}\label{subsec2.3}
In this subsection, we extend the above conclusions to equations with higher-order interactions. Let $q$ be any positive integer, take
\begin{align*}
f(x,\mu)=&A\left(\int_{\mathbb{R}^d\times\cdots\times \mathbb{R}^d}\kappa(x,y_1,\dots, y_q)\mu({\rm d}y_1)\cdots\mu({\rm d}y_q) \right),\notag\\
g(x,\mu)=&B\left(\int_{\mathbb{R}^d\times\cdots\times \mathbb{R}^d}\zeta(x,y_1,\dots, y_q) \mu({\rm d}y_1)\cdots\mu({\rm d}y_q) \right).
\end{align*}
Suppose that $a$ still satisfies Assumption \eqref{A2.2}, and the other terms all satisfy the global Lipschitz conditions.

\begin{assumption}\label{A2.5}
There are constants $L_4>0$ such that 
\begin{align*}
\vert A(x)-A(\bar{x})\vert\vee \vert B(x)-B(\bar{x})\vert\le L_4\vert x-\bar{x}\vert& ,\\
\vert \kappa(x,y_1,\dots, y_q)-\kappa(\bar{x},\bar{y}_1\dots,\bar{y}_q)\vert\vee \vert \zeta(x,y_1,\dots, y_q)-\zeta(\bar{x},\bar{y}_1\dots,\bar{y}_q)\vert\le& L_4\left(\vert x-\bar{x}\vert+\sum_{m=1}^q\vert y_j-\bar{y}_j\vert\right),\\
\vert A(0)\vert \vee\vert \kappa(0,\dots,0)\vert\vee \vert B(0)\vert\vee\vert \zeta(0,\dots,0)\vert\le& L_4.
\end{align*}
for all $x, y,\bar{x},\bar{y}\in\mathbb{R}^d$.
\end{assumption}

\begin{remark}\label{remark2.4}
Assumption \ref{A2.5} implies the following growth conditions:
\begin{align*}
\vert A(x)\vert\vee\vert B(x)\vert\le&  L_4\left(1+\vert x\vert\right),\\
\vert\kappa(x,y_1,\dots, y_q)\vert\vee\vert \zeta(x,y_1,\dots, y_q)\vert\le& L_4\left(1+\vert x\vert+\sum_{m=1}^q\vert y_j\vert\right).
\end{align*}
\end{remark}
\begin{remark} \label{f to F}\rm
According to Assumption \ref{A2.5}, using the Kantorovich-Rubinstein dual representation of the Wasserstein distance $\mathbb{W}_1$ and the fact that $\mathbb{W}_1(\mu,\sigma)\le \mathbb{W}_2(\mu,\sigma)$, we can also verify that
\begin{align*}
\vert f(x,\mu)-f(y,\sigma)\vert\le&C\big(\vert x-y\vert+\mathbb{W}_2(\mu,\sigma)\big),\\
\vert g(x,\mu)-g(y,\sigma)\vert\le&C\big(\vert x-y\vert+\mathbb{W}_2(\mu,\sigma)\big),\\
\vert f(0,\delta_0)\vert+\vert g(0,\delta_0)\vert\le& C
\end{align*}
for all $x,y\in\mathbb{R}^d$ and all $\mu,\sigma\in\mathcal{P}_2(\mathbb{R}^d)$. In fact, without loss of generality, we assume that 
\begin{align*}
f(x,\mu)=&A\left(\int_{\mathbb{R}^d\times \mathbb{R}^d}\kappa(x,y_1, y_2)\mu({\rm d}y_1)\mu({\rm d}y_2) \right),
\end{align*}
then 
\begin{align*}
\vert f(x,\mu)-f(y,\sigma)\vert\le&\vert f(x,\mu)-f(x,\sigma)\vert+\vert f(x,\sigma)-f(y,\sigma)\vert\notag\\
\le&L_4\left\vert \int_{\mathbb{R}^d\times \mathbb{R}^d}\kappa(x,y_1, y_2)\mu({\rm d}y_1)\mu({\rm d}y_2) -\int_{\mathbb{R}^d\times \mathbb{R}^d}\kappa(x,y_1, y_2)\sigma({\rm d}y_1)\sigma({\rm d}y_2)\right\vert\notag\\
&+L_4\left\vert \int_{\mathbb{R}^d\times \mathbb{R}^d}\kappa(x,y_1, y_2)\sigma({\rm d}y_1)\sigma({\rm d}y_2) -\int_{\mathbb{R}^d\times \mathbb{R}^d}\kappa(y,y_1, y_2)\sigma({\rm d}y_1)\sigma({\rm d}y_2) \right\vert\notag\\
\le&L_4\left\vert \int_{\mathbb{R}^d\times \mathbb{R}^d}\kappa(x,y_1, y_2)\mu({\rm d}y_1)\mu({\rm d}y_2) -\int_{\mathbb{R}^d\times \mathbb{R}^d}\kappa(x,y_1, y_2)\mu({\rm d}y_1)\sigma({\rm d}y_2)\right\vert\notag\\
&+L_4\left\vert \int_{\mathbb{R}^d\times \mathbb{R}^d}\kappa(x,y_1, y_2)\mu({\rm d}y_1)\sigma({\rm d}y_2) -\int_{\mathbb{R}^d\times \mathbb{R}^d}\kappa(x,y_1, y_2)\sigma({\rm d}y_1)\sigma({\rm d}y_2)\right\vert\notag\\
&+L_4 \int_{\mathbb{R}^d\times \mathbb{R}^d}\left\vert\kappa(x,y_1, y_2)-\kappa(y,y_1, y_2)\right\vert\sigma({\rm d}y_1)\sigma({\rm d}y_2)\notag\\
=:&\Gamma_1+\Gamma_1+\Gamma_3.
\end{align*}
For any given $x\in\mathbb R^d$, define $\phi_x(z)=\frac{1}{L_4}\int_{\mathbb R^d}\kappa(x,y_1,z)\mu({\rm d}y_1)$, then by Assumption \ref{A2.5}, 
\begin{align*}
\vert \phi_x(z)-\phi_x(\bar z)\vert=&\frac{1}{L_4}\left\vert \int_{\mathbb R^d}\kappa(x,y_1,z)\mu({\rm d}y_1)-\int_{\mathbb R^d}\kappa(x,y_1,\bar z)\mu({\rm d}y_1)\right\vert\\
\le &\frac{1}{L_4} \int_{\mathbb R^d}\left\vert\kappa(x,y_1,z)-\kappa(x,y_1,\bar z)\right\vert\mu({\rm d}y_1)\\
\le & \vert z-\bar z\vert.
\end{align*}
Hence $\phi_x(z)$ is Lipschitz continuous w.r.t. $z$, using the Kantorovich-Rubinstein dual representation of the Wasserstein distance $\mathbb{W}_1$ and the fact that $\mathbb{W}_1(\mu,\sigma)\le \mathbb{W}_2(\mu,\sigma)$, one has
\begin{align*}
\Gamma_1=L_4^2\left\vert \int_{\mathbb{R}^d}\phi_x(y_2)\mu({\rm d}y_2) -\int_{\mathbb{R}^d}\phi_x(y_2)\sigma({\rm d}y_2)\right\vert\le L_4^2\mathbb{W}_1(\mu,\sigma)\le L_4^2\mathbb{W}_2(\mu,\sigma).
\end{align*}
Similarly, we can also deduce that $\Gamma_2\le L_4^2\mathbb{W}_2(\mu,\sigma).$ Moreover,
$$\Gamma_3\le L_4^2 \int_{\mathbb{R}^d\times \mathbb{R}^d}\left\vert x-y\right\vert\sigma({\rm d}y_1)\sigma({\rm d}y_2)= L_4^2\vert x-y\vert,$$ hence
$$\vert f(x,\mu)-f(y,\sigma)\vert\le L_4^2\left(\vert x-y\vert+\mathbb{W}_2(\mu,\sigma)\right).$$
Other assertions can be similarly proved. Hence, we can also obtain the existence, uniqueness and boundedness of the solution of equation \eqref{NIPS} with coefficients taking the form of \eqref{coefficient-3.1} and \eqref{coefficient-3.2} following \cite[Theorem 3.3]{G.dosReis2019}.
\end{remark}

\begin{theorem}\label{theorem_x^i-x^i,K,N-3}
Let $p\ge 2$, it holds under Assumptions \ref{A2.1}, \ref{A2.3}, and \ref{A2.5} that
\begin{align*}
\max_{1\le i\le N}\mathbb{E}\left(\sup_{0\le t\le T}\vert x_t^i-X_t^{i,N,\delta}\vert^p\right)\le& C\left\{N^{-\frac{p}{2}}+\sup_{0\le t\le T}(t-\delta_t)^{\frac{p}{2}}+\max_{1\le i\le N}\mathbb{E}\left(\sup_{0\le t\le T}\vert a(X_{t}^{i,N,\delta})-\tilde{a}(X_{t}^{i,N,\delta})\vert^p\right)\right\}.
\end{align*}
\end{theorem}
The procedure for the proof of this theorem is also similar to that of Theorem \ref{theorem_x^i-x^i,K,N}, but for the reliability of the conclusion we still give a detailed proof. Define
\begin{align*}
\hat{\kappa}(x_t^i):=\int_{\mathbb{R}^d\times\cdots\times \mathbb{R}^d}\kappa(x_t^i,y_1,\dots, y_q)\mu_t^{x^i}({\rm d}y_1)\cdots\mu_t^{x^i}({\rm d}y_q),\\
\hat{\zeta}(x_t^i):=\int_{\mathbb{R}^d\times\cdots\times \mathbb{R}^d}\zeta(x_t^i,y_1,\dots, y_q)\mu_t^{x^i}({\rm d}y_1)\cdots\mu_t^{x^i}({\rm d}y_q).
\end{align*}
For $i=1,2,\dots, N$ and $x^i\in\mathbb{R}^d$, define
\begin{align*}
\hat{\kappa}^{N}(x^i):=\frac{1}{N^q}\sum_{j_1,\dots, j_q=1}^N\kappa(x^i,x^{j_1},\dots, x^{j_q}),\quad \hat{\zeta}^{N}(x^i):=\frac{1}{N^q}\sum_{j_1,\dots, j_q=1}^N\zeta(x^i,x^{j_1},\dots, x^{j_q}),
\end{align*}

\begin{lemma}\label{kappa-kappa^N-3}
Let $p\ge2$, for any $t\in[0,T]$, it holds under Assumptions \ref{A2.1} and \ref{A2.5} that
\begin{align*}
\mathbb{E}\left\vert \hat{\kappa}(x_t^i)-\hat{\kappa}^{N}(x_t^i)\right\vert^p\vee\mathbb{E}\left\vert \hat{\zeta}(x_t^i)-\hat{\zeta}^{N}(x_t^i)\right\vert^p\le CN^{-\frac{p}{2}}.
\end{align*}
\end{lemma}
\begin{proof}
According to Jensen's inequality, we have 
\begin{align*}
&\mathbb{E}\vert \hat{\kappa}(x_t^i)-\hat{\kappa}^{N}(x_t^i)\vert^p\notag\\
=&\mathbb{E}\left\vert \frac{1}{N^q}\sum_{j_1,\dots, j_q=1}^N\int_{\mathbb{R}^d\times\cdots\times\mathbb{R}^d}\left(\kappa(x_t^i,y_1,\dots,y_q)-\kappa(x_t^i,x_t^{j_1},\dots,x_t^{j_q})\right)\mu_t^{x^i}({\rm d}y_1)\cdots\mu_t^{x^i}({\rm d}y_q)\right\vert^p\notag\\
\le&CN^{1-q-p}\sum_{j_1,\dots, j_{q-1}=1}^N\mathbb{E}\left\vert \sum_{j_q=1}^N\int_{\mathbb{R}^d\times\cdots\times\mathbb{R}^d}\left(\kappa(x_t^i,y_1,\dots,y_q)-\kappa(x_t^i,x_t^{j_1},\dots,x_t^{j_q})\right)\mu_t^{x^i}({\rm d}y_1)\cdots\mu_t^{x^i}({\rm d}y_q)\right\vert^p\notag\\
\le&CN^{1-q-p}\sum_{j_1,\dots, j_{q-1}=1}^N\mathbb{E}\left\vert \sum_{j_q\in\{i,j_1,\dots, j_{q-1}\}}\int_{\mathbb{R}^d\times\cdots\times\mathbb{R}^d}\left(\kappa(x_t^i,y_1,\dots,y_q)-\kappa(x_t^i,x_t^{j_1},\dots,x_t^{j_q})\right)\mu_t^{x^i}({\rm d}y_1)\cdots\mu_t^{x^i}({\rm d}y_q)\right\vert^p\notag\\
&+CN^{1-q-p}\sum_{j_1,\dots, j_{q-1}=1}^N\mathbb{E}\left\vert \sum_{j_q\neq i,j_1,\dots,j_{q-1}}\int_{\mathbb{R}^d\times\cdots\times\mathbb{R}^d}\left(\kappa(x_t^i,y_1,\dots,y_q)-\kappa(x_t^i,x_t^{j_1},\dots,x_t^{j_q})\right)\mu_t^{x^i}({\rm d}y_1)\cdots\mu_t^{x^i}({\rm d}y_q)\right\vert^p\notag\\
\end{align*}
For $p>2$, according to the tower property, using the conditional Rosenthal inequality (Lemma \ref{conditional Rosenthal-type inequality}), and the H\"older inequality, we can deduce that
{\small
\begin{align*}
&\mathbb{E}\left\vert  \sum_{j_q\neq i,j_1,\dots,j_{q-1}}\int_{\mathbb{R}^d\times\cdots\times\mathbb{R}^d}\left(\kappa(x_t^i,y_1,\dots,y_q)-\kappa(x_t^i,x_t^{j_1},\dots,x_t^{j_q})\right)\mu_t^{x^i}({\rm d}y_1)\cdots\mu_t^{x^i}({\rm d}y_q)\right\vert^p\notag\\
=&\mathbb{E}\left\{\mathbb{E}\left[\left\vert \sum_{j_q\neq i,j_1,\dots,j_{q-1}}\int_{\mathbb{R}^d\times\cdots\times\mathbb{R}^d}\left(\kappa(x_t^i,y_1,\dots,y_q)-\kappa(x_t^i,x_t^{j_1},\dots,x_t^{j_q})\right)\mu_t^{x^i}({\rm d}y_1)\cdots\mu_t^{x^i}({\rm d}y_q)\right\vert^p\Big\vert \sigma(x_t^i,x_t^{j_1},\dots,x_t^{j_{q-1}})\right]\right\}\notag\\
\le&\mathbb{E}\left\{ \sum_{j_q\neq i,j_1,\dots,j_{q-1}}\mathbb{E}\left[\left\vert\int_{\mathbb{R}^d\times\cdots\times\mathbb{R}^d}\left(\kappa(x_t^i,y_1,\dots,y_q)-\kappa(x_t^i,x_t^{j_1},\dots,x_t^{j_q})\right)\mu_t^{x^i}({\rm d}y_1)\cdots\mu_t^{x^i}({\rm d}y_q)\right\vert^p\Big\vert \sigma(x_t^i,x_t^{j_1},\dots,x_t^{j_{q-1}})\right]\right\}\notag\\
&+\mathbb{E}\left\{\left( \sum_{j_q\neq i,j_1,\dots,j_{q-1}}\mathbb{E}\left[\left\vert\int_{\mathbb{R}^d\times\cdots\times\mathbb{R}^d}\left(\kappa(x_t^i,y_1,\dots,y_q)-\kappa(x_t^i,x_t^{j_1},\dots,x_t^{j_q})\right)\mu_t^{x^i}({\rm d}y_1)\cdots\mu_t^{x^i}({\rm d}y_q)\right\vert^2\Big\vert \sigma(x_t^i,x_t^{j_1},\dots,x_t^{j_{q-1}})\right]\right)^{\frac{p}{2}}\right\}\notag\\
\le&\left(1+N^{\frac{p}{2}-1}\right) \sum_{j_q\neq i,j_1,\dots,j_{q-1}}\mathbb{E}\left\vert\int_{\mathbb{R}^d\times\cdots\times\mathbb{R}^d}\left(\kappa(x_t^i,y_1,\dots,y_q)-\kappa(x_t^i,x_t^{j_1},\dots,x_t^{j_q})\right)\mu_t^{x^i}({\rm d}y_1)\cdots\mu_t^{x^i}({\rm d}y_q)\right\vert^p,
\end{align*}}
where $\sigma(x_t^i,x_t^{j_1},\dots,x_t^{j_{q-1}})$ denotes the $\sigma$-algebra generated by $x_t^i,x_t^{j_1},\dots,x_t^{j_{q-1}}$. Then, applying the H\"older inequality, Assumption \ref{A2.5} and Lemma \ref{x_t_bounded}, yields that
{\small
\begin{align}\label{F_1-F_1^N}
&\mathbb{E}\vert \hat{\kappa}(x_t^i)-\hat{\kappa}^{N}(x_t^i)\vert^p\notag\\
\le&Cq^{p-1}N^{1-q-p}\sum_{j_1,\dots, j_{q-1}=1}^N\sum_{j_q\in\{i,j_1,\dots, j_{q-1}\}}\mathbb{E}\left\vert \int_{\mathbb{R}^d\times\cdots\times\mathbb{R}^d}\left(\kappa(x_t^i,y_1,\dots,y_q)-\kappa(x_t^i,x_t^{j_1},\dots,x_t^{j_q})\right)\mu_t^{x^i}({\rm d}y_1)\cdots\mu_t^{x^i}({\rm d}y_q)\right\vert^p\notag\\
&+CN^{-q-\frac{p}{2}}\left(1+N^{1-\frac{p}{2}}\right) \sum_{j_1,\dots, j_{q-1}=1}^N\sum_{j_q\neq i,j_1,\dots,j_{q-1}}\mathbb{E}\left\vert\int_{\mathbb{R}^d\times\cdots\times\mathbb{R}^d}\left(\kappa(x_t^i,y_1,\dots,y_q)-\kappa(x_t^i,x_t^{j_1},\dots,x_t^{j_q})\right)\mu_t^{x^i}({\rm d}y_1)\cdots\mu_t^{x^i}({\rm d}y_q)\right\vert^p\notag\\
\le&CN^{-q-\frac{p}{2}}\left(1+2N^{1-\frac{p}{2}}\right) \sum_{j_1,\dots, j_{q}=1}^N\mathbb{E}\left\vert\int_{\mathbb{R}^d\times\cdots\times\mathbb{R}^d}\left(\kappa(x_t^i,y_1,\dots,y_q)-\kappa(x_t^i,x_t^{j_1},\dots,x_t^{j_q})\right)\mu_t^{x^i}({\rm d}y_1)\cdots\mu_t^{x^i}({\rm d}y_q)\right\vert^p\notag\\
\le&CN^{-q-\frac{p}{2}}\left(1+2N^{1-\frac{p}{2}}\right)\sum_{j_1,\dots, j_{q}=1}^N\mathbb{E}\int_{\mathbb{R}^d\times\cdots\times\mathbb{R}^d}\left\vert\kappa(x_t^i,y_1,\dots,y_q)-\kappa(x_t^i,x_t^{j_1},\dots,x_t^{j_q})\right\vert^p\mu_t^{x^i}({\rm d}y_1)\cdots\mu_t^{x^i}({\rm d}y_q)\notag\\
\le&CN^{-q-\frac{p}{2}}\left(1+2N^{1-\frac{p}{2}}\right)\sum_{j_1,\dots, j_{q}=1}^N\mathbb{E}\int_{\mathbb{R}^d\times\cdots\times\mathbb{R}^d}\left(\vert y_1-x_t^{j_1}\vert^p +\cdots+\vert y_q-x_t^{j_q}\vert^p\right)\mu_t^{x^i}({\rm d}y_1)\cdots\mu_t^{x^i}({\rm d}y_q)\notag\\
\le&CN^{-q-\frac{p}{2}}\left(1+2N^{1-\frac{p}{2}}\right)\sum_{j_1,\dots, j_{q}=1}^N\mathbb{E}\left(\mathbb{E}\vert x_t^i\vert^p+\vert x_t^{j_1}\vert^p +\cdots+\vert x_t^{j_q}\vert^p\right)\notag\\
\le&CN^{-\frac{p}{2}}\left(1+2N^{1-\frac{p}{2}}\right)\notag\\
\le&CN^{-\frac{p}{2}}.
\end{align}}
Analogously we can also get
\begin{align*}
\mathbb{E}\left\vert \hat{\zeta}(x_t^i)-\hat{\zeta}^{N}(x_t^i)\right\vert^p\le CN^{-\frac{p}{2}}
\end{align*}
from Assumption \ref{A2.3}. The assertion for $p=2$ also holds following the H\"older inequality.
\end{proof}
\begin{lemma}\label{X_t-X_delta_t-3}
Let $p\ge 2$, for any $i\in\{1,2,\dots, N\}$ and $t\in[0,T]$, it holds under Assumptions \ref{A2.1}, \ref{A2.3} and \ref{A2.5} that
\begin{align*}
\mathbb{E} \vert X_t^{i,N,\delta}-X_{\delta_t}^{i,N,\delta}\vert^p\le C(t-\delta_t)^{\frac{p}{2}}.
\end{align*}
\end{lemma}
\begin{proof}
Since $\delta_t=[t/\delta]\delta, \delta>0$, it is easy to know that $0\le \delta_t\le t$, then
 \begin{align*}
X_t^{i,N,\delta}=&X_{\delta_t}^{i,N,\delta}+\int_{\delta_t}^t \tilde{a}(X_{\delta_s}^{i,N,\delta}){\rm d}s+\int_{\delta_t}^tA\left(\hat{\kappa}^{N}(X_{\delta_s}^{i,N,\delta})\right){\rm d}s+\int_{\delta_t}^tB\left(\hat{\zeta}^{N}(X_{\delta_s}^{i,N,\delta})\right){\rm d}W_s^i.
\end{align*}
On the other hand, for $s\in[\delta_t,t]$, since $\delta_s=\delta_t$, using Jensen's inequality, the B-D-G inequality, and Remark \ref{remark2.4}, we have
 \begin{align*}
&\mathbb{E}\vert X_t^{i,N,\delta}-X_{\delta_t}^{i,N,\delta}\vert^p\\
\le &C(t-\delta_t)^p\mathbb{E}\left\vert \tilde{a}(X_{\delta_t}^{i,N,\delta})\right\vert^p+C(t-\delta_t)^p\mathbb{E}\left\vert A\left(\hat{\kappa}^{N}(X_{\delta_s}^{i,N,\delta})\right)\right\vert^p+C(t-\delta_t)^{\frac{p}{2}}\mathbb{E}\left\vert B\left(\hat{\zeta}^{N}(X_{\delta_t}^{i,N,\delta})\right)\right\vert^p\\
\le &C(t-\delta_t)^p\mathbb{E}\left\vert \tilde{a}(X_{\delta_t}^{i,N,\delta})\right\vert^p+C(t-\delta_t)^p\left(1+\mathbb{E}\left\vert \hat{\kappa}^{N}(X_{\delta_s}^{i,N,\delta})\right\vert^p\right)+C(t-\delta_t)^{\frac{p}{2}}\left(1+\mathbb{E}\left\vert \hat{\zeta}^{N}(X_{\delta_s}^{i,N,\delta})\right\vert^p\right)\\
=&C(t-\delta_t)^p\mathbb{E}\left\vert \tilde{a}(X_{\delta_t}^{i,N,\delta})\right\vert^p+C(t-\delta_t)^p\left(1+\mathbb{E}\left\vert \frac{1}{N^q}\sum_{j_1,\dots, j_q=1}^N\kappa(X_{\delta_s}^{i,N,\delta},X_{\delta_s}^{j_1,N,\delta},\dots, X_{\delta_s}^{j_q,N,\delta})\right\vert^p\right)\\
&+C(t-\delta_t)^{\frac{p}{2}}\left(1+\mathbb{E}\left\vert \frac{1}{N^q}\sum_{j_1,\dots, j_q=1}^N\zeta(X_{\delta_s}^{i,N,\delta},X_{\delta_s}^{j_1,N,\delta},\dots, X_{\delta_s}^{j_q,N,\delta})\right\vert^p\right)\\
\le &C(t-\delta_t)^p\mathbb{E}\left\vert \tilde{a}(X_{\delta_t}^{i,N,\delta})\right\vert^p+C(t-\delta_t)^p\left(1+\frac{1}{N^q}\sum_{j_1,\dots, j_q=1}^N\mathbb{E}\left\vert \kappa(X_{\delta_s}^{i,N,\delta},X_{\delta_s}^{j_1,N,\delta},\dots, X_{\delta_s}^{j_q,N,\delta})\right\vert^p\right)\\
&+C(t-\delta_t)^{\frac{p}{2}}\left(1+\mathbb{E}\left\vert \frac{1}{N^q}\sum_{j_1,\dots, j_q=1}^N\zeta(X_{\delta_s}^{i,N,\delta},X_{\delta_s}^{j_1,N,\delta},\dots, X_{\delta_s}^{j_q,N,\delta})\right\vert^p\right)\\
\le &C(t-\delta_t)^p\mathbb{E}\left\vert \tilde{a}(X_{\delta_t}^{i,N,\delta})\right\vert^p+C\left((t-\delta_t)^p+(t-\delta_t)^{\frac{p}{2}}\right)\left(1+\frac{1}{N^q}\sum_{j_1,\dots, j_q=1}^N\mathbb{E}\left(1+\vert X_{\delta_s}^{i,N,\delta}\vert^p+\sum_{m=1}^q\vert X_{\delta_s}^{j_m,N,\delta}\vert^p\right)\right)\\
\le &C(t-\delta_t)^p\mathbb{E}\left\vert \tilde{a}(X_{\delta_t}^{i,N,\delta})\right\vert^p+C\left((t-\delta_t)^p+(t-\delta_t)^{\frac{p}{2}}\right)\left(1+\mathbb{E}\vert X_{\delta_s}^{i,N,\delta}\vert^p+\frac{1}{N^q}\sum_{j_1,\dots, j_q=1}^N\sum_{m=1}^q\mathbb{E}\vert X_{\delta_s}^{j_m,N,\delta}\vert^p\right)\\
\le &C(t-\delta_t)^{\frac{p}{2}}\left((t-\delta_t)^{\frac{p}{2}}\mathbb{E}\left\vert \tilde{a}(X_{\delta_t}^{i,N,\delta})\right\vert^p+C\left(1+(t-\delta_t)^{\frac{p}{2}}\right)\left(1+\mathbb{E}\vert X_{\delta_s}^{i,N,\delta}\vert^p\right)\right).
\end{align*}
Since $t-\delta_t\le t\le T$, applying Assumption \ref{A2.3}, it holds that
 \begin{align*}
\mathbb{E}\vert X_t^{i,N,\delta}-X_{\delta_t}^{i,N,\delta}\vert^p\le C(t-\delta_t)^{\frac{p}{2}}.
\end{align*}
\end{proof}

\begin{proof}[Proof of Theorem \ref{theorem_x^i-x^i,K,N-3}]
The proof follows similar steps to Theorem \ref{theorem_x^i-x^i,K,N}, and we only present the different parts here. In fact, according to the definition of $\hat{\kappa}, \hat{\zeta}, \hat{\kappa}^N, \hat{\zeta}^N$ in this subsection, 
we still have
\begin{align*}
e_t^{i,N}=x_t^i-X_t^{i,N,\delta}&=\int_0^t\left(a(x_s^i)-\tilde{a}(X_{\delta_s}^{i,N,\delta})\right){\rm d}s+\int_0^t\left( A\left(\hat{\kappa}(x_s^i)\right)-A\left(\hat{\kappa}^{N}(X_{\delta_s}^{i,N,\delta})\right)\right){\rm d}s\notag\\
&\quad+\int_0^t \left(B\left(\hat{\zeta}(x_s^i)\right)-B\left(\hat{\zeta}^{N}(X_{\delta_s}^{i,N,\delta})\right)\right){\rm d}W_s^i
\end{align*}
for any $i=1,2,\dots,N$. Then, applying the It\^o formula, the Cauchy-Schwarz inequality, and Assumption \ref{A2.5}, it also gives \eqref{e^p}. Moreover, since $a$ also satisfies Assumption \ref{A2.1}, the estimation for $J_1$ is the same, but the estimates of $J_2$ and $J_3$ are slightly different. Actually, it can be deduced from the Young inequality and Assumption \ref{A2.5} that
\begin{align}\label{J_tilde_2}
J_2=&pL_4\int_0^t \vert e_s^{i,N}\vert^{p-1}\left\vert \hat{\kappa}(x_s^i)-\hat{\kappa}^{N}(X_{\delta_s}^{i,N,\delta})\right\vert{\rm d}s\notag\\
\le&pL_4\int_0^t \vert e_s^{i,N}\vert^{p-1}\left(\vert\hat{\kappa}(x_s^i)-\hat{\kappa}^{N}(x_s^i)\vert+\vert\hat{\kappa}^{N}(x_s^i)-\hat{\kappa}^{N}(X_s^{i,N,\delta})\vert+\vert\hat{\kappa}^{N}(X_s^{i,N,\delta})-\hat{\kappa}^{N}(X_{\delta_s}^{i,N,\delta})\vert\right){\rm d}s\notag\\
\le&C\int_0^t\vert e_s^{i,N}\vert^p{\rm d}s+C\int_0^t\vert \hat{\kappa}(x_s^i)-\hat{\kappa}^{N}(x_s^i)\vert^p{\rm d}s+\frac{C}{N^q}\sum_{j_1,\cdots, j_q=1}^N\int_0^t\left(\vert e_s^{i,N}\vert^p+\sum_{m=1}^q\vert e_s^{j_m,N}\vert^p\right){\rm d}s\notag\\
&+\frac{C}{N^q}\sum_{j_1,\cdots, j_q=1}^N\int_0^t\left(\vert X_s^{i,N,\delta}-X_{\delta_s}^{i,N,\delta}\vert^p+\sum_{m=1}^q\vert X_s^{j_m,N,\delta}-X_{\delta_s}^{j_m,N,\delta}\vert^p\right){\rm d}s,
\end{align}
and
\begin{align}\label{J_tilde_3}
J_3=&\frac{p(p-1)L_4^2}{2}\int_0^t \vert e_s^{i,N}\vert^{p-2}\left\vert \hat{\zeta}(x_s^i)-\hat{\zeta}^{N}(X_{\delta_s}^{i,N,\delta})\right\vert^2{\rm d}s\notag\\
\le&C\int_0^t\vert e_s^{i,N}\vert^p{\rm d}s+C\int_0^t\vert \hat{\kappa}(x_s^i)-\hat{\kappa}^{N}(x_s^i)\vert^p{\rm d}s+\frac{C}{N^q}\sum_{j_1,\cdots, j_q=1}^N\int_0^t\left(\vert e_s^{i,N}\vert^p+\sum_{m=1}^q\vert e_s^{j_m,N}\vert^p\right){\rm d}s\notag\\
&+\frac{C}{N^q}\sum_{j_1,\cdots, j_q=1}^N\int_0^t\left(\vert X_s^{i,N,\delta}-X_{\delta_s}^{i,N,\delta}\vert^p+\sum_{m=1}^q\vert X_s^{j_m,N,\delta}-X_{\delta_s}^{j_m,N,\delta}\vert^p\right){\rm d}s.
\end{align}
Substituting \eqref{J_1}, \eqref{J_tilde_2}, and \eqref{J_tilde_3} into \eqref{e^p}, it gives that
\begin{align*}
\vert e_t^{i,N}\vert^p\le&C\int_0^t \vert e_s^{i,N}\vert^p{\rm d}s+C\int_0^t (1+\vert X_s^{i,N,\delta}\vert^r+\vert X_{\delta_s}^{i,N,\delta}\vert^r)^p\vert X_s^{i,N,\delta}-X_{\delta_s}^{i,N,\delta}\vert^p{\rm d}s\notag\\
&+\int_0^t \vert a(X_{\delta_s}^{i,N,\delta})-\tilde{a}(X_{\delta_s}^{i,N,\delta})\vert^p{\rm d}s+C\int_0^t\vert \hat{\kappa}(x_s^i)-\hat{\kappa}^{N}(x_s^i)\vert^p{\rm d}s+\frac{C}{N^q}\sum_{j_1,\cdots, j_q=1}^N\sum_{m=1}^q\int_0^t\vert e_s^{j_m,N}\vert^p{\rm d}s\notag\\
&+\frac{C}{N^q}\sum_{j_1,\cdots, j_q=1}^N\int_0^t\left(\vert X_s^{i,N,\delta}-X_{\delta_s}^{i,N,\delta}\vert^p+\sum_{m=1}^q\vert X_s^{j_m,N,\delta}-X_{\delta_s}^{j_m,N,\delta}\vert^p\right){\rm d}s+C\int_0^t \left\vert \hat{\zeta}(x_s^i)-\hat{\zeta}^{N}(x_s^i)\right\vert^p{\rm d}s
+J_4.
\end{align*}
Thus, similar to the derivation of \eqref{Esupe^p}, one can also obtain 
\begin{align*}
\mathbb{E}\left(\sup_{0\le t\le T_1}\vert e_t^{i,N}\vert^p\right)\le &C\int_0^{T_1}\mathbb{E}\vert e_t^{i,N}\vert^p{\rm d}t+C\int_0^{T_1}\left(\mathbb{E} (1+\vert X_t^{i,N,\delta}\vert^r+\vert X_{\delta_t}^{i,N,\delta}\vert^r)^{2p}\mathbb{E}\vert X_t^{i,N,\delta}-X_{\delta_t}^{i,N,\delta}\vert^{2p}\right)^{1/2}{\rm d}t\notag\\
&+\int_0^{T_1}\mathbb{E} \vert a(X_{\delta_t}^{i,N,\delta})-\tilde{a}(X_{\delta_t}^{i,N,\delta})\vert^p{\rm d}t+C\int_0^{T_1}\mathbb{E} \vert \hat{\kappa}(x_t^i)-\hat{\kappa}^{N}(x_t^i)\vert^p{\rm d}t\notag\\
&+C\int_0^{T_1}\mathbb{E}\vert X_t^{i,N,\delta}-X_{\delta_t}^{i,N,\delta}\vert^p{\rm d}t+C\int_0^{T_1}\mathbb{E} \left\vert \hat{\zeta}(x_t^i)-\hat{\zeta}^{N}(x_t^i)\right\vert^p{\rm d}t.
\end{align*}
Combining Lemmas \ref{kappa-kappa^N-3}, \ref{X_t-X_delta_t-3} and Assumption \ref{A2.3}, the assertion holds true again.
\end{proof}


\section{Dimension-independent rate of propagation of chaos}\label{sec4}
As we explained in the subsection \ref{subsec2.2}, if we take $\tilde{a}=a$ and $\delta=t$ in \eqref{IPS-2}, which implies $\delta_t=t$, let $X_t^{i,N}:=X_t^{i,N,t}$, then \eqref{IPS-2} becomes
\begin{align*}
{\rm d}X_t^{i,N}=& a(X_{t}^{i,N}){\rm d}t+f(X_{t}^{i,N},\hat\mu_{t}^{X,N}){\rm d}t+g(X_{t}^{i,N},\hat\mu_{t}^{X,N}){\rm d}W^i_t,
\end{align*}
which is the corresponding interacting particle system of \eqref{NIPS}, and the following quantitive convergence rate of PoC follows as an immediate by-product of Theorem \ref{theorem_x^i-x^i,K,N}.

\begin{theorem}\label{theorem_x^i-X^i,N}
Let $p\ge 2$, it holds under Assumptions \ref{A2.1}-\ref{A2.3} that
\begin{align*}
\mathbb{E}\left(\sup_{0\le t\le T}\vert x_t^i-X_t^{i,N}\vert^p\right)\le CN^{-\frac{p}{2}},
\end{align*}
where $x_t^i$ and $X_t^{i,N}$ are the solutions of MV-SDEs \eqref{NIPS} with coefficient given by \eqref{coefficient-1} and its corresponding interacting particle system, respectively. 
\end{theorem}
In alignment with Theorems \ref{theorem_x^i-x^i,K,N-2} and \ref{theorem_x^i-x^i,K,N-3}, when the coefficients in equation \eqref{NIPS} are specified by \eqref{coefficient-2.1}-\eqref{coefficient-2.2} or \eqref{coefficient-3.1}-\eqref{coefficient-3.2}, the same convergence rate of PoC can be achieved. To avoid redundancy, we omit the explicit statement of the corresponding theorem here. 

In fact, in the case where the coefficients of MV-SDEs are linear in measure, i.e.,
 \begin{align*}
f(x,\mu)=\int_{\mathbb{R}^d}\kappa(x,y)\mu({\rm d}y),~g(x,\mu)=\int_{\mathbb{R}^d}\zeta(x,y)\mu({\rm d}y),
\end{align*}
\cite{Sznitman1991book} showed that $\mathbb{E}\left(\sup_{0\le t\le T}\vert x_t^i-X_t^{i,N}\vert\right)\le CN^{-\frac{1}{2}}$. While in the case of the coefficients of MV-SDEs are Lipschitz continuous in measure, the rate of strong PoC deteriorates with the dimension $d$, for instance, the PoC rate under the $L^2$-norm presented in \cite{Carmona2018, Xingyuan Chen2023Arxiv, G.dosReis2022} is
  \begin{align*}
\sup_{1\le i\le N}\mathbb{E}\left( \sup_{0\le t\le T}\vert x_t^i-X_t^{i,N}\vert^2\right)\le C
\begin{cases}
N^{-1/2}& {\rm if} ~~d<4,\\
N^{-1/2}\log(N)& {\rm if} ~~ d=4,\\
N^{-2/d}& {\rm if} ~~ d>4,
\end{cases}
\end{align*}
and the rate under the $L^p$-norm for $p>0$ and $q>p$ given in \cite{Guo2024ESAIM,YatingLiu2024} is
 \begin{align}\label{L^p rate}
\sup_{1\le i\le N}\sup_{0\le t\le T}\mathbb{E}\vert x_t^i-X_t^{i,N}\vert^p\le C
\begin{cases}
N^{-1/2}+N^{-(q-p)/q}& {\rm if} ~~p>d/2~~{\rm and}~~ q\neq 2p,\\
N^{-1/2}\log(N)+N^{-(q-p)/q}& {\rm if} ~~p=d/2~~{\rm and}~~ q\neq 2p,\\
N^{-2/d}+N^{-(q-p)/q}& {\rm if} ~~p\in[2, d/2)~~{\rm and}~~ q\neq d/(d-p).
\end{cases}
\end{align}
It can be observed that the PoC rate presented above depends on the dimension of the system and decays at approximately $O(N^{-1/2p})$  when $p>d/2$.  At this point, if the coefficients simultaneously satisfy some stronger regularity conditions, the rate $1/2$ was previously established in \cite[Lemma 5.1]{F.Delarue2019} and \cite[Theorem 4.2]{Szpruch2021}. However, in this work, we treat some special cases of nonlinear MV-SDEs with coefficients do not claim strong regularity, this is obviously a different set of sufficient conditions compared to the existing results.

\section{Strong convergence rate of the time-stepping method}\label{The time-stepping scheme}
The convergence of the time-stepping method has been well studied in many literatures, such as tamed EM method (\cite{G.dosReis2022}), truncated EM method (\cite{Guo2024ESAIM}), adaptive EM method (\cite{Reisinger2022}), and so on. The convergence rate of the time-stepping scheme is not the concern of this paper, but in order to give a computable numerical scheme for MV-SDEs, here we introduce the tamed EM method and utilize it in the next section to compute the solution of the test equations, which then verifies the convergence order of the PoC. Here we only give the numerical scheme of equation \eqref{NIPS} when the coefficients are taken to be \eqref{coefficient-1} and verify Assumption \ref{A2.3}. Finally, the order of convergence of the numerical method is obtained according to Theorem \ref{theorem_x^i-x^i,K,N}.

In the following, let $M\in \mathbb{Z}_+$, $\Delta=T/M$ be the step size, we take $\delta=\Delta$ and 
\begin{align*}
\tilde{a}(x)=\frac{a(x)}{1+\sqrt{\Delta}\vert a(x)\vert},
\end{align*}
in \eqref{IPS-2}, and establish the convergence of the tamed EM method for the particle system \eqref{NIPS} with coefficients given by \eqref{coefficient-1}. It is easy to known that
\begin{align}\label{f_D<f}
\vert \tilde{a}(x)\vert\le \min\{\vert a(x)\vert, \Delta^{-1/2}\},
\end{align} 
and $\delta_t=[t/\Delta]\Delta$, then \eqref{IPS-2} becomes
\begin{align}\label{NIPS-3.2}
{\rm d}X_t^{i,N,\Delta}=&\frac{a(X_{\delta_t}^{i,N,\Delta})}{1+\sqrt{\Delta}\vert a(X_{\delta_t}^{i,N,\Delta})\vert}{\rm d}t+A\left(\frac{1}{N}\sum_{j=1}^N\kappa(X_{\delta_t}^{i,N,\Delta},X_{\delta_t}^{j,N,\Delta})\right){\rm d}t+B\left(\frac{1}{N}\sum_{j=1}^N\zeta(X_{\delta_t}^{i,N,\Delta},X_{\delta_t}^{j,N,\Delta}) \right){\rm d}W^i_t,
\end{align}
which is in accordance with the continuous version of the tamed EM method for \eqref{NIPS} with coefficients given by \eqref{coefficient-1}.

Under the settings above, we can verify Assumption \ref{A2.3} easily.
\begin{lemma}\label{numerical solution bounded-2}
Let $p\ge 2$, under Assumptions \ref{A2.1}-\ref{A2.3}, it holds that
\begin{align*}
\sup_{1\le i\le N}\left\{\mathbb{E}\left(\sup_{0\le t\le T}\vert X_t^{i,N,\Delta}\vert^p\right)\vee\mathbb{E}\left(\sup_{0\le t\le T}\vert \tilde{a}(X_t^{i,N,\Delta})\vert^p\right)\right\}\le C.
\end{align*}
\end{lemma}
\begin{proof}
For any $i\in\{1,\dots, N\}$ and $p>2$, using the It\^o formula to $\vert X_t^{i,N,\Delta}\vert^p$ with $X_t^{i,N,\Delta}$ satisfies \eqref{NIPS-3.2}, we have
\begin{align*}
\vert X_t^{i,N,\Delta}\vert^p\le&\vert x_0^i\vert^p+p\int_{0}^t \vert X_s^{i,N,\Delta}\vert^{p-2}\left\langle X_s^{i,N,\Delta},\tilde{a}(X_{\delta_s}^{i,N,\Delta})\right\rangle{\rm d}s\notag\\
&+p\int_{0}^t \vert X_s^{i,N,\Delta}\vert^{p-1}\left\vert A\left(\frac{1}{N}\sum_{j=1}^N\kappa(X_{\delta_s}^{i,N,\Delta},X_{\delta_s}^{j,N,\Delta})\right)\right\vert{\rm d}s\notag\\
&+p\int_{0}^t \vert X_s^{i,N,\Delta}\vert^{p-1}\left\vert B\left(\frac{1}{N}\sum_{j=1}^N\zeta(X_{\delta_s}^{i,N,\Delta},X_{\delta_s}^{j,N,\Delta})\right)\right\vert{\rm d}W_s^i\notag\\
&+\frac{p(p-1)}{2}\int_0^{t}\vert X_s^{i,N,\Delta}\vert^{p-2} \left\vert B\left(\frac{1}{N}\sum_{j=1}^N\zeta(X_{\delta_s}^{i,N,\Delta},X_{\delta_s}^{j,N,\Delta})\right)\right\vert^2{\rm d}s.
\end{align*}
Applying the definition of $\tilde{a}$ and Remark \ref{remark1}, using Young's inequality, one can deduce that 
\begin{align*}
\vert X_t^{i,N,\Delta}\vert^p\le&\vert x_0^i\vert^p+p\int_{0}^t \vert X_s^{i,N,\Delta}\vert^{p-2}\left\langle X_s^{i,N,\Delta}-X_{\delta_s}^{i,N,\Delta},\tilde{a}(X_{\delta_s}^{i,N,\Delta})\right\rangle{\rm d}s\notag\\
&+p\int_{0}^t \frac{\vert X_s^{i,N,\Delta}\vert^{p-2}\left\langle X_{\delta_s}^{i,N,\Delta},a(X_{\delta_s}^{i,N,\Delta})\right\rangle}{1+\sqrt{\Delta}\vert a(X_{\delta_s}^{i,N,\Delta})\vert}{\rm d}s\notag\\
&+C\frac{1}{N}\sum_{j=1}^N\int_{0}^t\left( \vert X_s^{i,N,\Delta}\vert^{p}+\vert X_{\delta_s}^{i,N,\Delta}\vert^{p}+\vert X_{\delta_s}^{j,N,\Delta}\vert^{p}\right){\rm d}s\notag\\
&+C\frac{1}{N}\sum_{j=1}^N\int_{0}^t \vert X_s^{i,N,\Delta}\vert^{p-1}\left(1+\vert X_{\delta_s}^{i,N,\Delta}\vert+\vert X_{\delta_s}^{j,N,\Delta}\vert\right){\rm d}W_s^i.
\end{align*}
Then for any $T_1\in[0,T]$, since $\{X_t^{i,N,\Delta}\}_{1\le i\le N}$ are identically distributed, according to the Young and B-D-G inequalities, we have
\begin{align}\label{A_1+A_2+A_3-2}
\mathbb{E}\left(\sup_{0\le t\le T_1}\vert X_t^{i,N,\Delta}\vert^p\right)\le&\mathbb{E}\vert x_0^i\vert^p+p\mathbb{E}\sup_{0\le t\le T_1}\int_{0}^t \vert X_s^{i,N,\Delta}\vert^{p-2}\left\langle X_s^{i,N,\Delta}-X_{\delta_s}^{i,N,\Delta},\tilde{a}(X_{\delta_s}^{i,N,\Delta})\right\rangle{\rm d}s\notag\\
&+C\frac{1}{N}\sum_{j=1}^N\mathbb{E}\int_{0}^{T_1}\left( 1+\vert X_t^{i,N,\Delta}\vert^{p}+\vert X_{\delta_t}^{i,N,\Delta}\vert^{p}+\vert X_{\delta_t}^{j,N,\Delta}\vert^{p}\right){\rm d}t\notag\\
&+C\frac{1}{N}\sum_{j=1}^N\mathbb{E}\left(\int_{0}^{T_1} \vert X_t^{i,N,\Delta}\vert^{2(p-1)}\left(1+\vert X_{\delta_t}^{i,N,\Delta}\vert+\vert X_{\delta_t}^{j,N,\Delta}\vert\right)^2{\rm d}t\right)^{1/2}\notag\\
\le &C(1+\mathbb{E}\vert x_0^i\vert^p)+\frac{1}{2}\mathbb{E}\left(\sup_{0\le t\le T_1}\vert X_t^{i,N,\Delta}\vert^p\right)+C\int_{0}^{T_1}\mathbb{E}\left(\sup_{0\le s\le t} \vert X_s^{i,N,\Delta}\vert^{p}\right){\rm d}t\notag\\
&+p\mathbb{E}\sup_{0\le t\le T_1}\int_{0}^t \vert X_s^{i,N,\Delta}\vert^{p-2}\left\langle X_s^{i,N,\Delta}-X_{\delta_s}^{i,N,\Delta},\tilde{a}(X_{\delta_s}^{i,N,\Delta})\right\rangle{\rm d}s.
\end{align}
According to \eqref{NIPS-3.2} and Remark \ref{remark1}, recall that $\vert \tilde{a}\vert\le \Delta^{-1/2}$ and $\delta_s=\delta_t$ for $s\in [\delta_t,t]$, using the Young inequality once again, it holds that
\begin{align*}
&\mathbb{E}\sup_{0\le t\le T_1}\int_{0}^t \vert X_s^{i,N,\Delta}\vert^{p-2}\left\langle X_s^{i,N,\Delta}-X_{\delta_s}^{i,N,\Delta},\tilde{a}(X_{\delta_s}^{i,N,\Delta})\right\rangle{\rm d}s\\
\le&\mathbb{E}\int_{0}^{T_1} \vert X_t^{i,N,\Delta}\vert^{p-2}\vert \tilde{a}(X_{\delta_t}^{i,N,\Delta})\vert^2\Delta{\rm d}t\\
&+\mathbb{E}\int_{0}^{T_1}\vert X_t^{i,N,\Delta}\vert^{p-2} \left\vert  A\left(\frac{1}{N}\sum_{j=1}^N\kappa(X_{\delta_t}^{i,N,\Delta},X_{\delta_t}^{j,N,\Delta})\right)(t-\delta_t)\right\vert\vert \tilde{a}(X_{\delta_t}^{i,N,\Delta})\vert{\rm d}t\\
&+\mathbb{E}\int_{0}^{T_1}\vert X_t^{i,N,\Delta}\vert^{p-2} \left\vert B\left(\frac{1}{N}\sum_{j=1}^N\zeta(X_{\delta_t}^{i,N,\Delta},X_{\delta_t}^{j,N,\Delta})\right)(W_t^i-W_{\delta_t}^i)\right\vert\vert \tilde{a}(X_{\delta_t}^{i,N,\Delta})\vert{\rm d}t\\
\le&C\frac{1}{N}\sum_{j=1}^N\int_{0}^{T_1}  \left(1+\mathbb{E}\vert X_t^{i,N,\Delta}\vert^{p}+\mathbb{E}\vert X_{\delta_t}^{i,N,\Delta}\vert^p+\mathbb{E}\vert X_{\delta_t}^{j,N,\Delta}\vert^p\right){\rm d}t\\
&+ \mathbb{E}\int_{0}^{T_1}\left\vert B\left(\frac{1}{N}\sum_{j=1}^N\zeta(X_{\delta_t}^{i,N,\Delta},X_{\delta_t}^{j,N,\Delta})\right)(W_t^i-W_{\delta_t}^i)\right\vert^{\frac{p}{2}}\Delta^{-p/4}{\rm d}t\\
\le&C+C\int_{0}^{T_1}\mathbb{E}\left(\sup_{0\le s\le t} \vert X_s^{i,N,\Delta}\vert^p\right){\rm d}t.
\end{align*}
Substituting this inequality into \eqref{A_1+A_2+A_3-2}, it can be derived that
\begin{align*}
\mathbb{E}\left(\sup_{0\le t\le T_1}\vert X_t^{i,N,\Delta}\vert^p\right)\le C+C\int_{0}^{T_1} \mathbb{E}\left(\sup_{0\le s\le t}\vert X_s^{i,N,\Delta}\vert^p\right){\rm d}t.
\end{align*}
Using the Gronwall inequality, it gives
\begin{align*}
\mathbb{E}\left(\sup_{0\le t\le T_1}\vert X_t^{i,N,\Delta}\vert^p\right)\le C{\rm e}^{CT_1}, \forall T_1\in[0,T],
\end{align*}
therefore,
\begin{align*}
\sup_{1\le i\le N}\mathbb{E}\left(\sup_{0\le t\le T}\vert X_t^{i,N,\Delta}\vert^p\right)\le C.
\end{align*}
Similarly, we can also get the 2nd moment boundedness of the numerical solution. Moreover, since $\vert \tilde{a}(x)\vert\le \vert a(x)\vert$, using Remark \ref{remark1}, one has
\[
\mathbb{E}\left(\sup_{0\le t\le T}\vert \tilde{a}(X_t^{i,N,\Delta})\vert^p\right)\le \mathbb{E}\left(\sup_{0\le t\le T}\vert a(X_t^{i,N,\Delta})\vert^p\right)\le C+C\mathbb{E}\left(\sup_{0\le t\le T}\vert X_t^{i,N,\Delta}\vert^{p(q+1)}\right).
\]
Note that $\mathbb{E}\left(\sup_{t\in[0,T]}\vert X_t^{i,N,\Delta}\vert^p\right)\le C$ holds for all $p\ge 2$, then for any given $p\ge 2$ and $q>0$, we can also prove $\mathbb{E}\left(\sup_{t\in[0,T]}\vert X_t^{i,N,\Delta}\vert^{p(q+1)}\right)\le C$, which gives 
\[
\mathbb{E}\left(\sup_{0\le t\le T}\vert \tilde{a}(X_t^{i,N,\Delta})\vert^p\right)\le C.
\]
The proof is completed.
\end{proof} 

\begin{lemma}\label{a-a_tilde}
Let $p\ge 2$, under Assumption \ref{A2.1}, it holds that
\begin{align*}
\sup_{1\le i\le N}\mathbb{E}\left(\sup_{0\le t\le T}\vert a(X_{t}^{i,N,\Delta})-\tilde{a}(X_{t}^{i,N,\Delta})\vert^p\right)\le C\Delta^{\frac{p}{2}}.
\end{align*}
\end{lemma}

\begin{proof}
For any $i=1,2,\dots, N$, 
\begin{align*}
&\mathbb{E}\left(\sup_{0\le t\le T}\vert a(X_{t}^{i,N,\Delta})-\tilde{a}(X_{t}^{i,N,\Delta})\vert^p\right)=\mathbb{E}\left(\sup_{0\le t\le T}\left\vert a(X_{t}^{i,N,\Delta})-\frac{a(X_{t}^{i,N,\Delta})}{1+\sqrt{\Delta}\vert a(X_{t}^{i,N,\Delta})\vert}\right\vert^p\right)\\
\le& \mathbb{E}\left(\sup_{0\le t\le T}\vert a(X_{t}^{i,N,\Delta})\vert^p\right)\Delta^{\frac{p}{2}}\le C\Delta^{\frac{p}{2}}\left(1+ \mathbb{E}\sup_{0\le t\le T}\vert X_{t}^{i,N,\Delta}\vert^{p(q+1)}\right)\le C\Delta^{\frac{p}{2}}.
\end{align*}
\end{proof}


Recall that we take $\delta=\Delta$ in this subsection, which gives $\delta_t=[t/\Delta]\Delta$, then it is easy to known that $t-\delta_t\le \Delta$ for all $t\in [0,T]$. Substituting this statement and Lemma \ref{a-a_tilde} into Theorem \ref{theorem_x^i-x^i,K,N}, one can obtain the following convergence result of the proposed numerical method for the MV-SDEs \eqref{NIPS}.
\begin{theorem}
Let $p\ge 2$, it holds under Assumption \ref{A2.1} that
\begin{align*}
\sup_{1\le i\le N}\mathbb{E}\left(\sup_{0\le t\le T}\vert x_t^i-X_t^{i,N,\Delta}\vert^p\right)\le CN^{-\frac{p}{2}}+C\Delta^{\frac{p}{2}}.
\end{align*}
\end{theorem}

\section{Numerical simulation}\label{Numerical simulation}
We illustrate the convergence rate of the numerical simulation on the following examples. Since the convergence result for the tamed EM discretization is quite routine, we only give numerical tests on the convergence rate of PoC.  As the ``true” solution of the considered system is unknown, the errors for these examples are calculated in reference to a proxy solution given by an approximation with a larger number of particles. To be more specific, the strong convergence behavior on PoC is measured by
\begin{align*}
{\rm Propagation~of~chaos~error~(PoC-Error)}\approx\left(\frac{1}{\bar{N}}\sum_{i=1}^{\bar{N}} \vert X_T^{i,N,\Delta}-X_T^{i,\bar{N},\Delta}\vert^p\right)^{1/p}
\end{align*}
for a fixed time step size $\Delta=2^{-10}$, other parameters will be specified in the following. By performing a linear regression on the logarithm of the error and the number of particles using the least squares method, we give the images of the order of convergence. In this section, let $\phi(x)$ denote $\left(\phi(x_1),\phi(x_2),\dots,\phi(x_d)\right)^{\rm T}$ for $x\in\mathbb{R}^d$ and let $\phi(x)=\left(\phi(x_{ij})\right)_{d\times d}$ for $x\in\mathbb{R}^{d\times d}$.  

\begin{example}\label{exam1}\rm
First, we consider a scalar nonlinear MV-SDE \eqref{NIPS} with coefficients taking in the form \eqref{coefficient-1} under
\begin{align*}
a(x)=x-x^3,\quad A(x)=\frac{1}{1+{\rm e}^{-x}},\quad \kappa(x,y)=\arctan(x+y),\quad B(x)=\sin(x), \quad \zeta(x,y)=\sqrt{x^2+y^2},
\end{align*} 
let $x_0\sim N(0,1)$, it is equivalent to 
\begin{align*}
x_t=x_0+\int_0^t (x_s-x_s^3){\rm d}s+\int_0^t \frac{1}{1+{\rm exp}\{-\mathbb{E}[\arctan(z+x_s)]\big\vert_{z=x_s}\}}{\rm d}s+\int_0^t\sin\left(\mathbb{E}[\sqrt{z^2+x_s^2}]\big\vert_{z=x_s}\right){\rm d}W_s.
\end{align*}
It is easy to verify that $a(x), A(x),B(x), \kappa(x,y), \zeta(x,y)$ satisfy Assumptions \ref{A2.1}-\ref{A2.3}. The corresponding interacting particle system is  
\begin{align*}
X_t^{i,N}=&x_0+\int_0^t (X_s^{i,N}-(X_s^{i,N})^3){\rm d}s+\int_0^t \left(1+\exp\left\{-\frac{1}{N}\sum_{j=1}^N\arctan(X_s^{i,N}+X_s^{j,N})\right\}\right)^{-1}{\rm d}s\\
&+\int_0^t \sin\left(\frac{1}{N}\sum_{j=1}^N \sqrt{(X_s^{i,N})^2+(X_s^{j,N})^2}\right){\rm d}W_s^i,
\end{align*}
and the tamed EM scheme is given by
\begin{align*}
X_{t_{n+1}}^{i,N}=&X_{t_n}^{i,N}+\frac{X_{t_n}^{i,N}-(X_{t_n}^{i,N})^3}{1+\sqrt{\Delta}\vert X_{t_n}^{i,N}-(X_{t_n}^{i,N})^3\vert }\Delta+\left(1+\exp\left\{-\frac{1}{N}\sum_{j=1}^N\arctan(X_{t_n}^{i,N}+X_{t_n}^{j,N})\right\}\right)^{-1}\Delta\\
&+\sin\left(\frac{1}{N}\sum_{j=1}^N \sqrt{(X_{t_n}^{i,N})^2+(X_{t_n}^{j,N})^2}\right)\Delta W_n^i,
\end{align*}
where $t_n=n\Delta, n=0,1,\dots.$ The orders of convergence with respect to the number of particles $N$ are shown in Figure \ref{Fig.main1}. The proxy solution takes $N=2^{11}$. The numerical solutions take $\bar{N} =2^7, 2^8, 2^9, 2^{10}$.
 \begin{figure}[H]
\centering  
\includegraphics[width=0.5\textwidth]{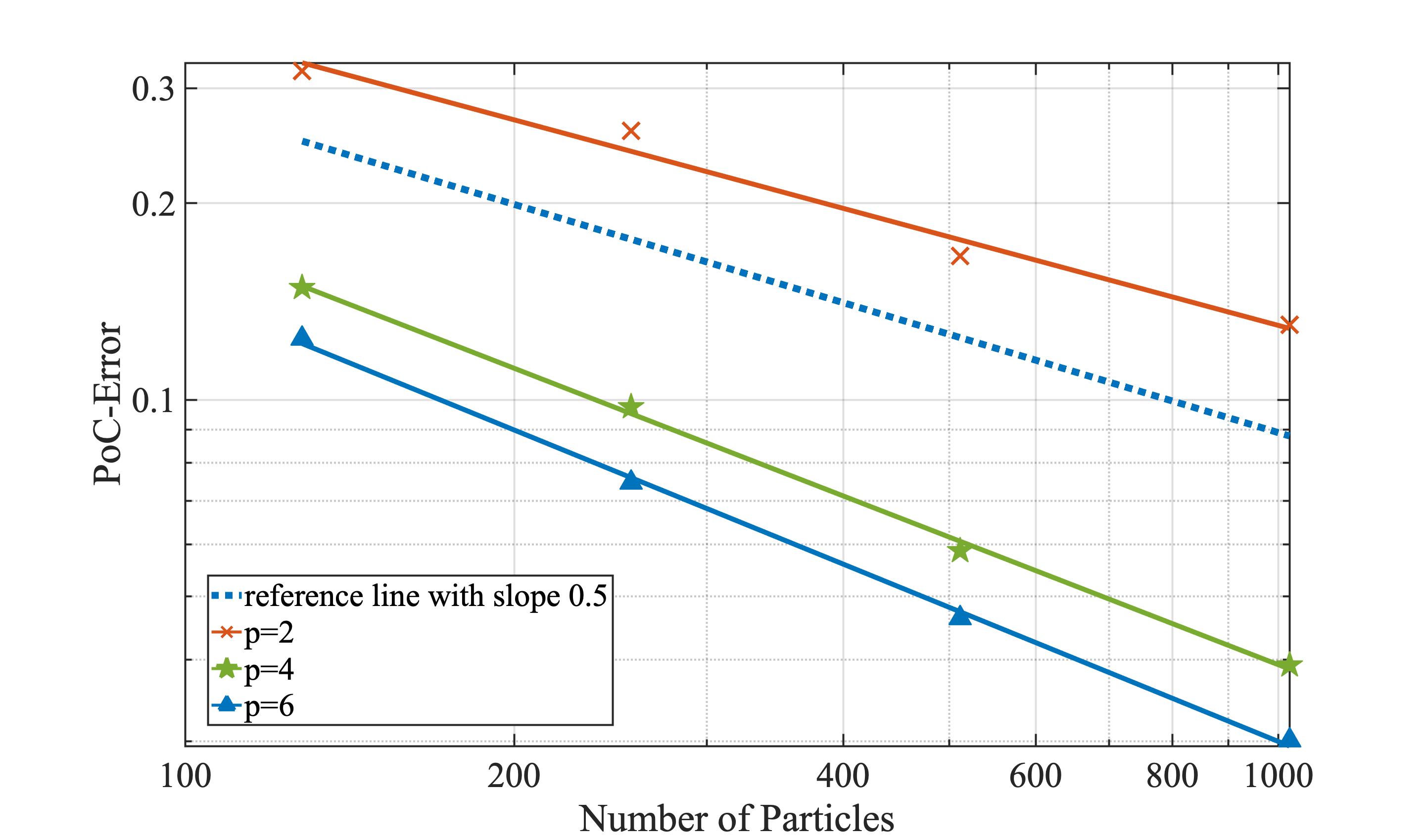}
\caption{$L^p$ convergence rate of PoC for Example \ref{exam1}.}
\label{Fig.main1}
\end{figure}
\end{example}

\begin{example}\label{exam2}\rm
In this example, we test the $L^p$ rate of PoC in different dimensions for $p=2, 4$ and compare the findings to the theoretical upper bounds established by \eqref{L^p rate}. For system \eqref{NIPS} with coefficients taking the form of \eqref{coefficient-1}, we make the following choices: Let $d \ge 2$, $x =(x_1, . . . , x_d)\in\mathbb{R}^d$, $(W_t)_{0\le t\le T}$ is a d-dim Brownian motion, the initial condition is a vector distributed according to $d$-independent $N(0, 1)$-random variables, the coefficients are taken to be $a(x)=x-x^3:=(x_1-x_1^3,\dots,x_d-x_d^3)^{\rm T}$,
\begin{align*}
\kappa(x,y)=\left(\begin{array}{c}{\rm sign}(x_1)\vert x_1+y_1\vert \\ {\rm sign}(x_2)\vert x_2+y_2\vert \\\vdots \\  {\rm sign}(x_d)\vert x_d+y_d\vert\end{array}\right),\quad \zeta(x,y)=\left(\begin{array}{cccc}\sqrt{x_1^2+y_1^2} & x_2& \cdots & x_d\\x_1 & \sqrt{x_2^2+y_2^2} & \cdots & x_d \\ \vdots & \vdots &  & \vdots\\x_1 & x_2 & \cdots & \sqrt{x_d^2+y_d^2}\end{array}\right),
\end{align*} 
and $A(x)=\sin(x), B(x)=\cos(x)$. The tamed EM scheme is given by
\begin{align*}
X_{t_{n+1}}^{i,N}=&X_{t_n}^{i,N}+\frac{(X_{t_n}^{i,N}-(X_{t_n}^{i,N})^3)\Delta}{1+\sqrt{\Delta}\vert X_{t_n}^{i,N}-(X_{t_n}^{i,N})^3\vert }+\sin\left(\frac{1}{N}\sum_{j=1}^N\kappa(X_{t_n}^{i,N},X_{t_n}^{j,N})\right)\Delta+\cos\left(\frac{1}{N}\sum_{j=1}^N \zeta(X_{t_n}^{i,N},X_{t_n}^{j,N})\right)\Delta W_n^i.
\end{align*}
We performed PoC rate tests in the sense of $L^2$, $L^4$ and $L^6$, respectively. When $p = 2$, we simulated the equations for dimensions $d=2$ (then $p>d/2$), $d=4 (p=d/2)$ and $d=6 (p\in[2, d/2))$, when $p = 4$, we simulated the equations for dimensions $d=4$ ($p>d/2$), $d=8 (p=d/2)$ and $d=10 (p\in[2, d/2))$, and when $p = 6$, we simulated the equations for dimensions $d=4, 12, 14$. It can be found in  Figures \ref{Fig.main2} and \ref{Fig.main3} that the PoC rates observed numerically are all around $\mathcal{O}(N^{-1/2})$ and do not depend on the dimension of the solution.
\begin{figure}
\centering  
\subfigure[d = 2]{
\label{figure_2.1}		
\includegraphics[width=0.32\textwidth]{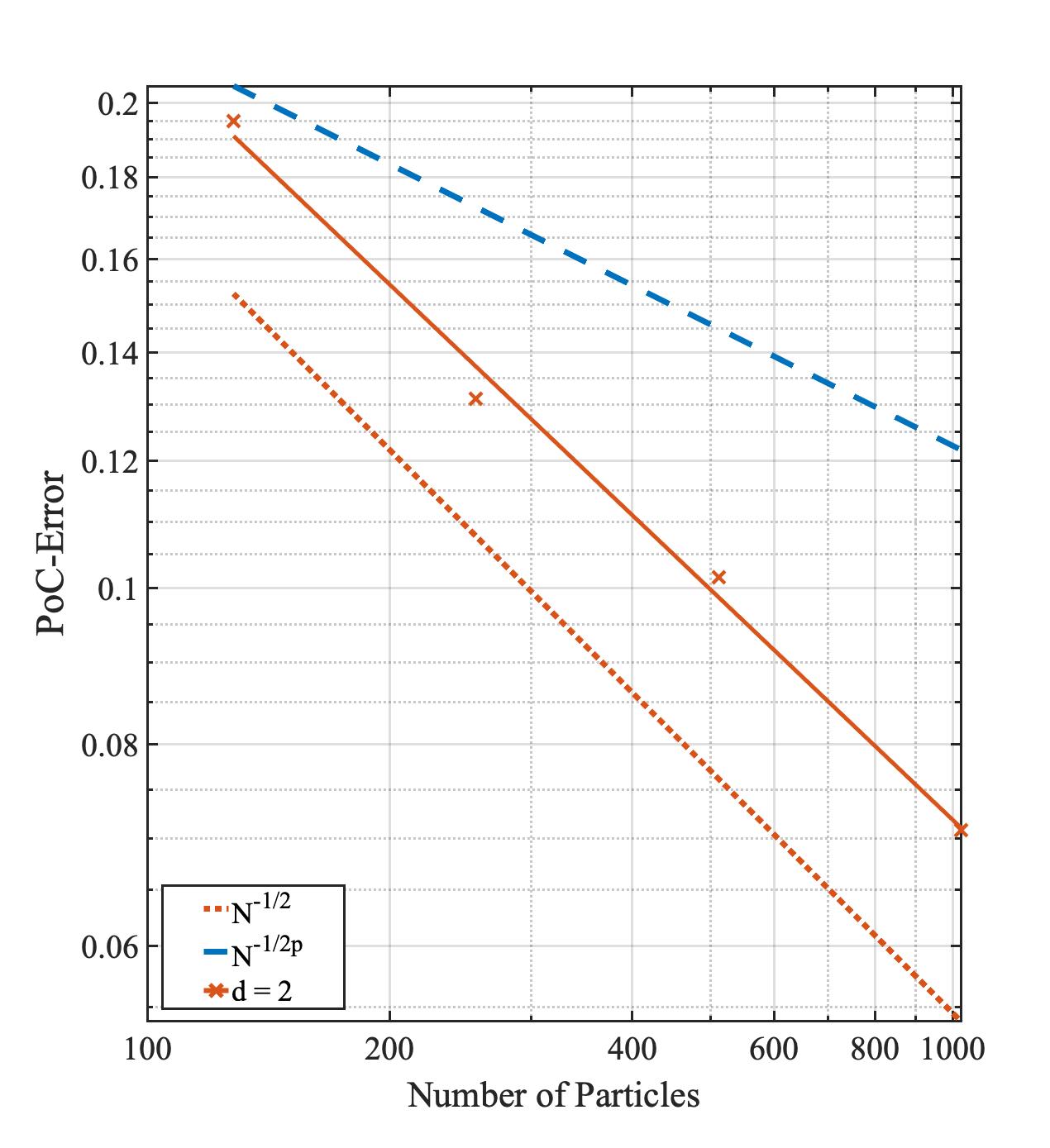}}
\subfigure[d = 4]{
\label{figure_2.2}		
\includegraphics[width=0.32\textwidth]{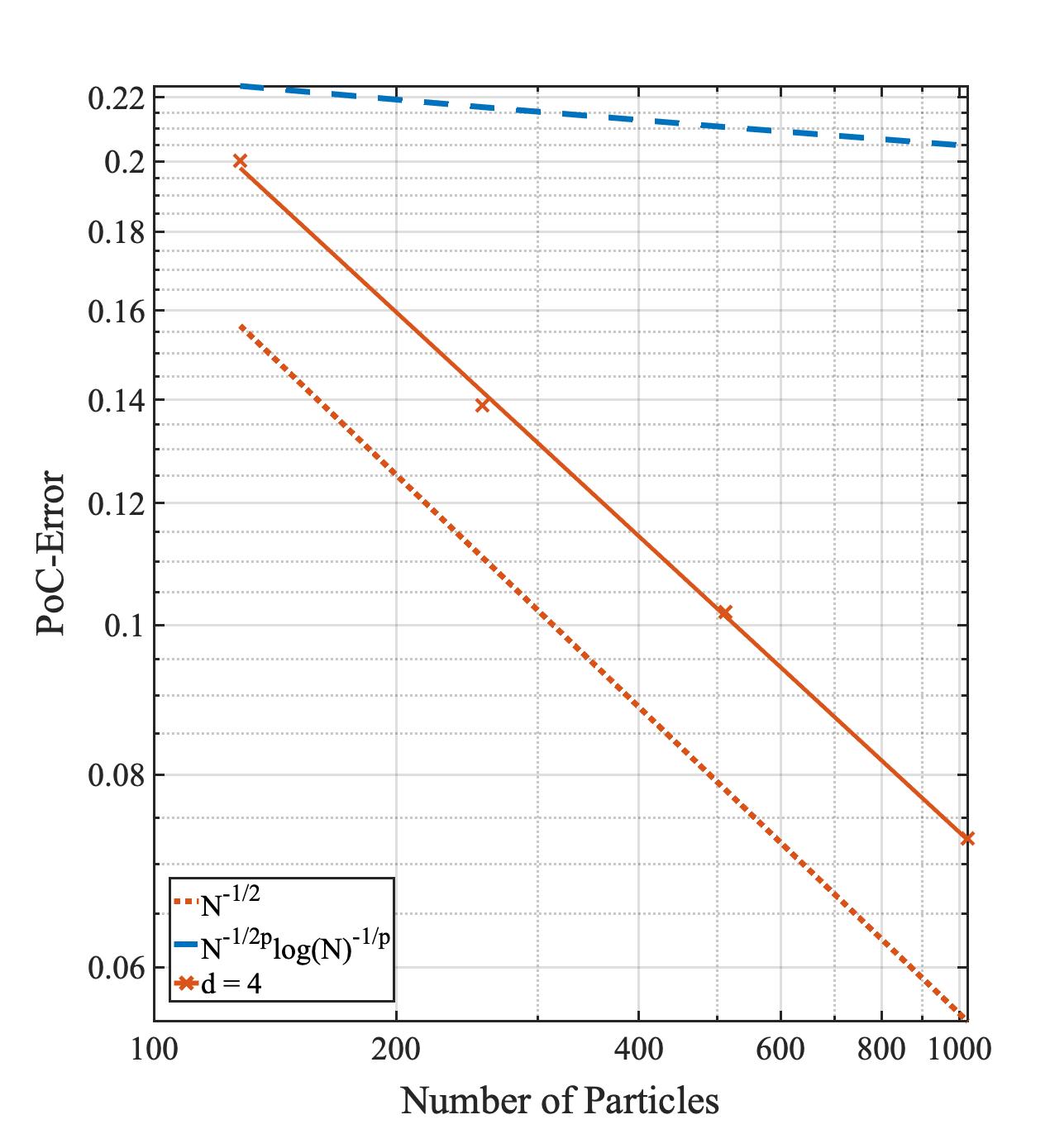}}
\subfigure[d = 6]{
\label{figure_2.3}		
\includegraphics[width=0.32\textwidth]{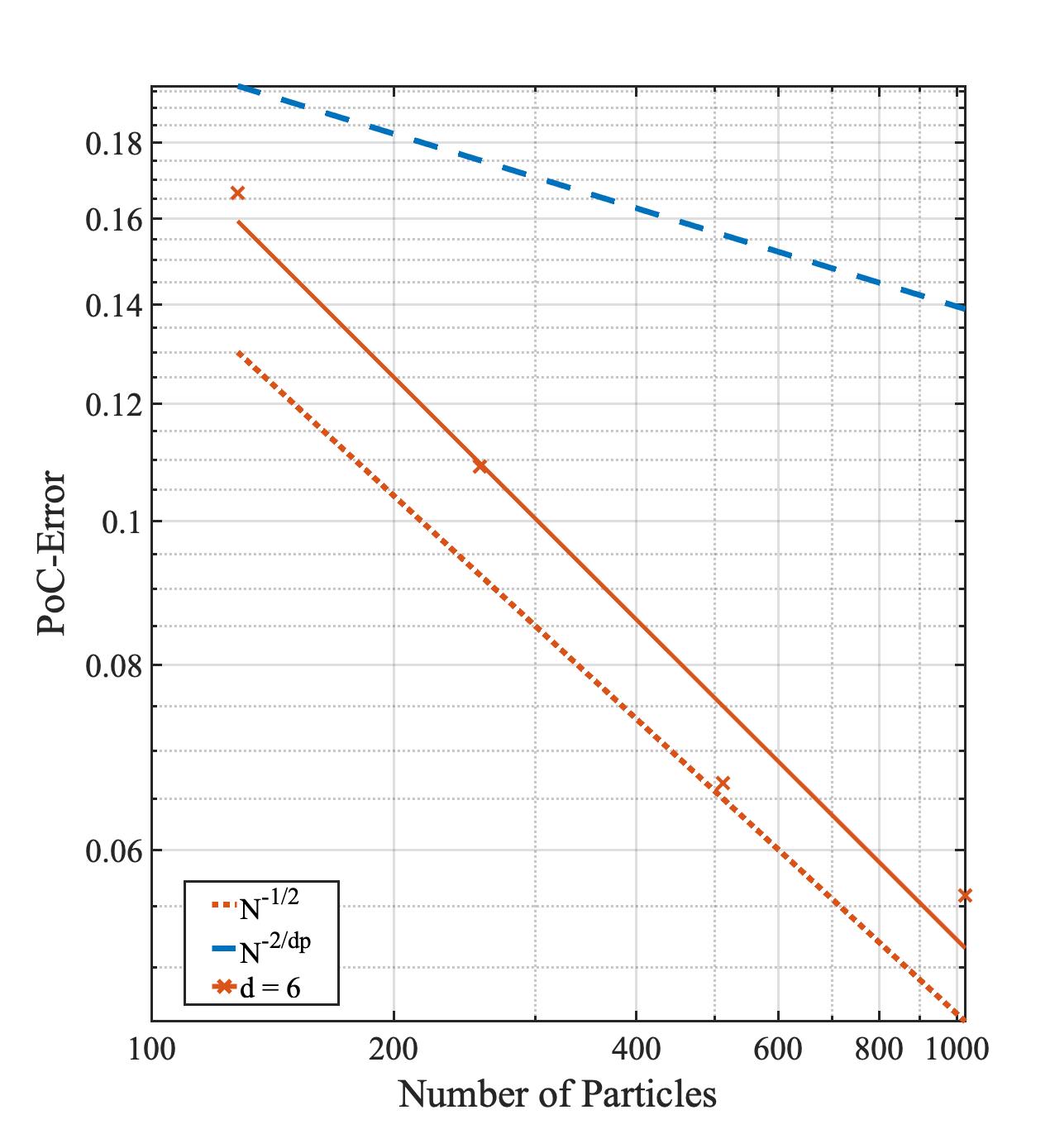}}
\caption{$L^2$ convergence rate of PoC for Example \ref{exam2}.}
\label{Fig.main2}
\end{figure}

\begin{figure}
\centering  
\subfigure[d = 6]{
\label{figure_3.1}		
\includegraphics[width=0.32\textwidth]{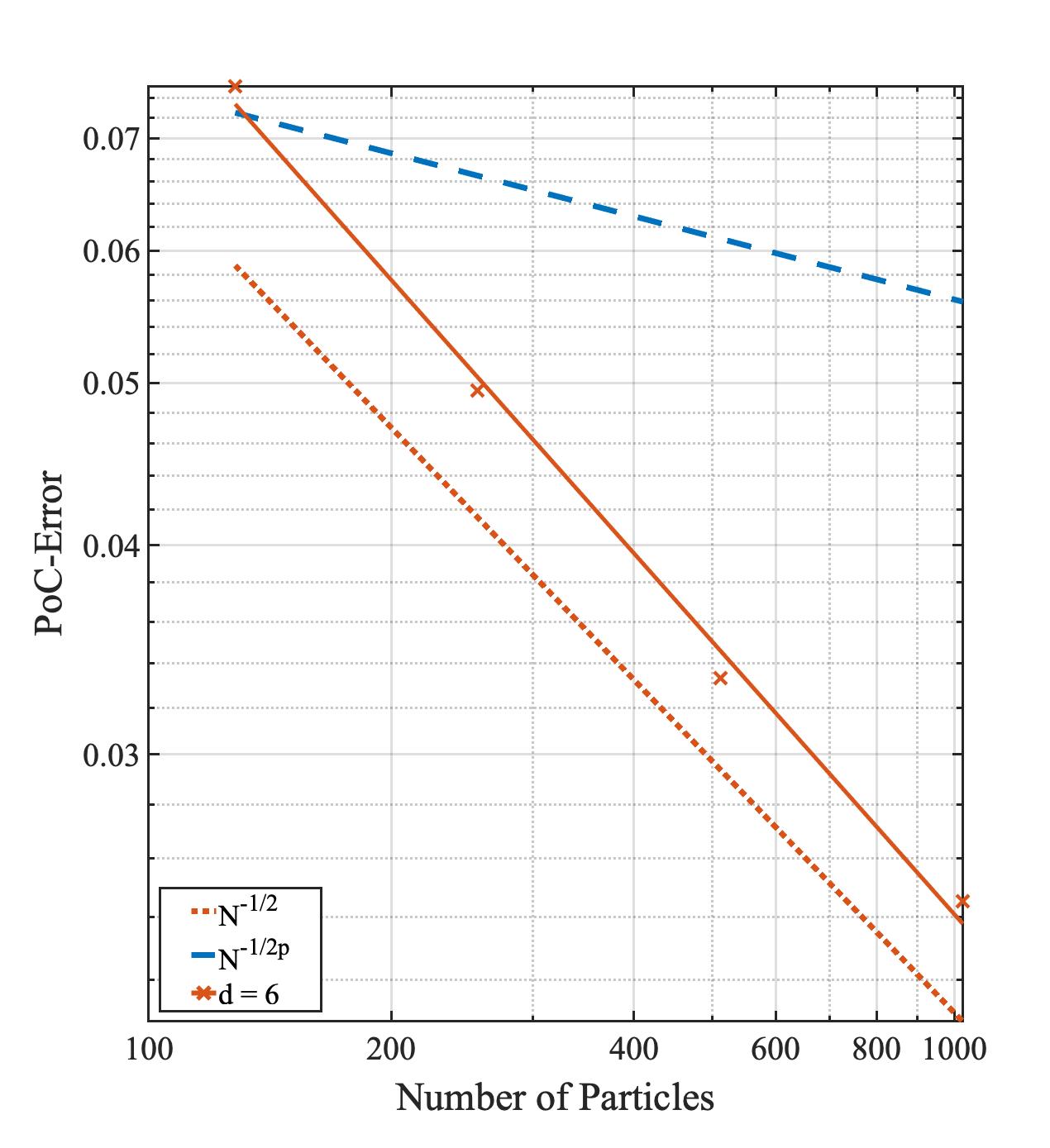}}
\subfigure[d = 8]{
\label{figure_3.2}		
\includegraphics[width=0.32\textwidth]{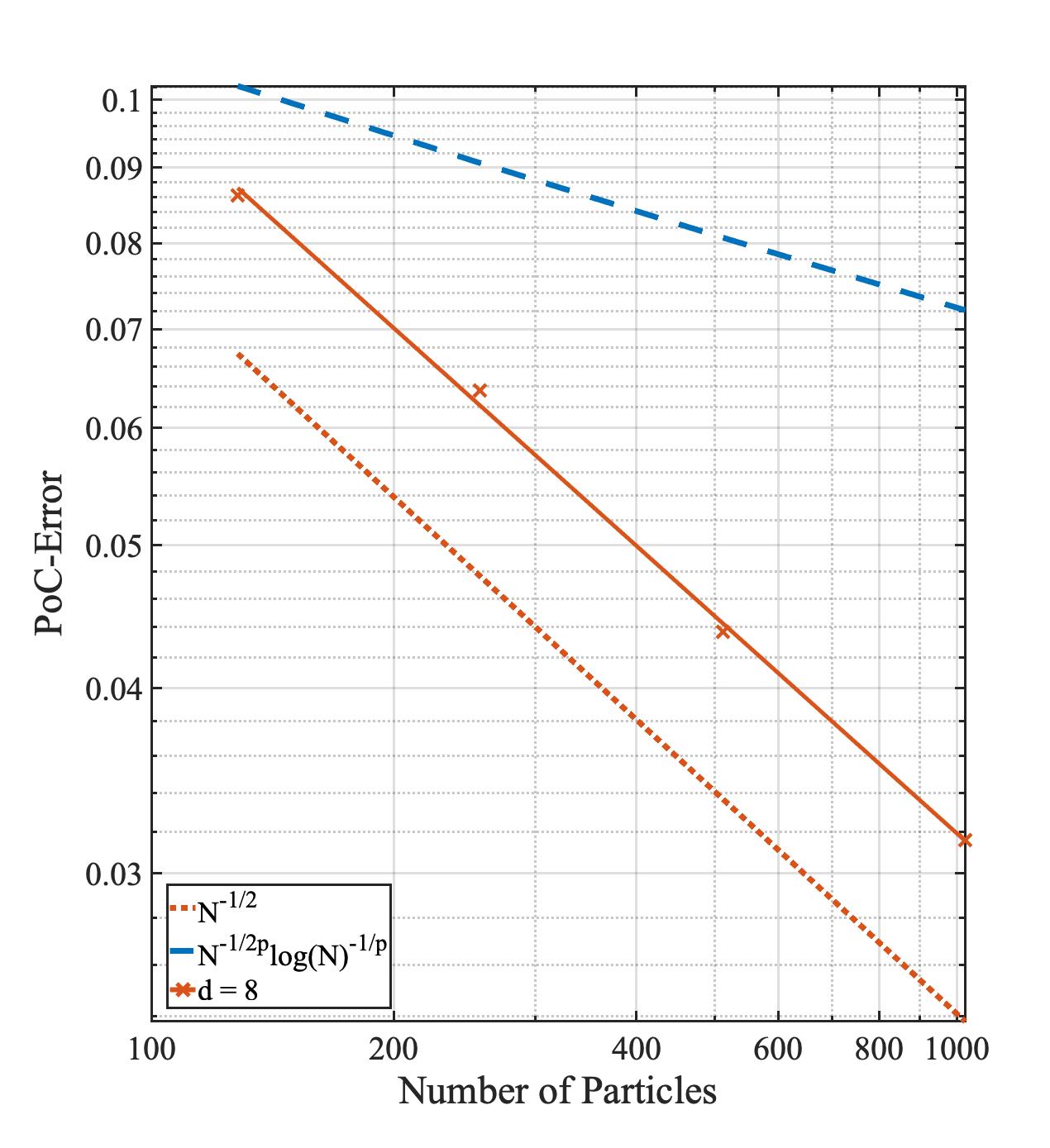}}
\subfigure[d = 10]{
\label{figure_3.3}		
\includegraphics[width=0.32\textwidth]{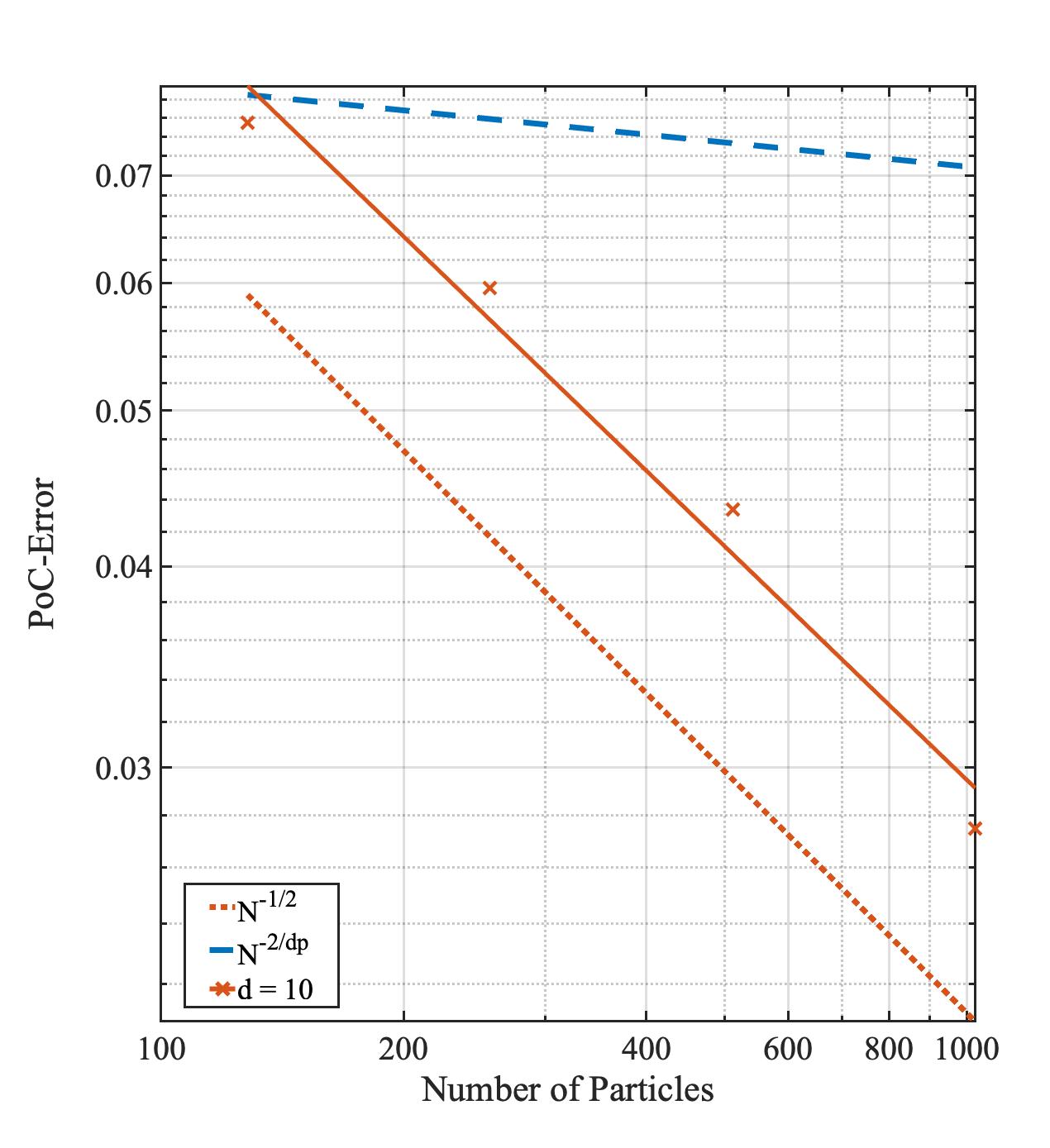}}
\caption{$L^4$ convergence rate of PoC for Example \ref{exam2}.}
\label{Fig.main3}
\end{figure}

\end{example}
\begin{figure}
\centering  
\subfigure[d = 2]{
\label{figure_4.1}		
\includegraphics[width=0.32\textwidth]{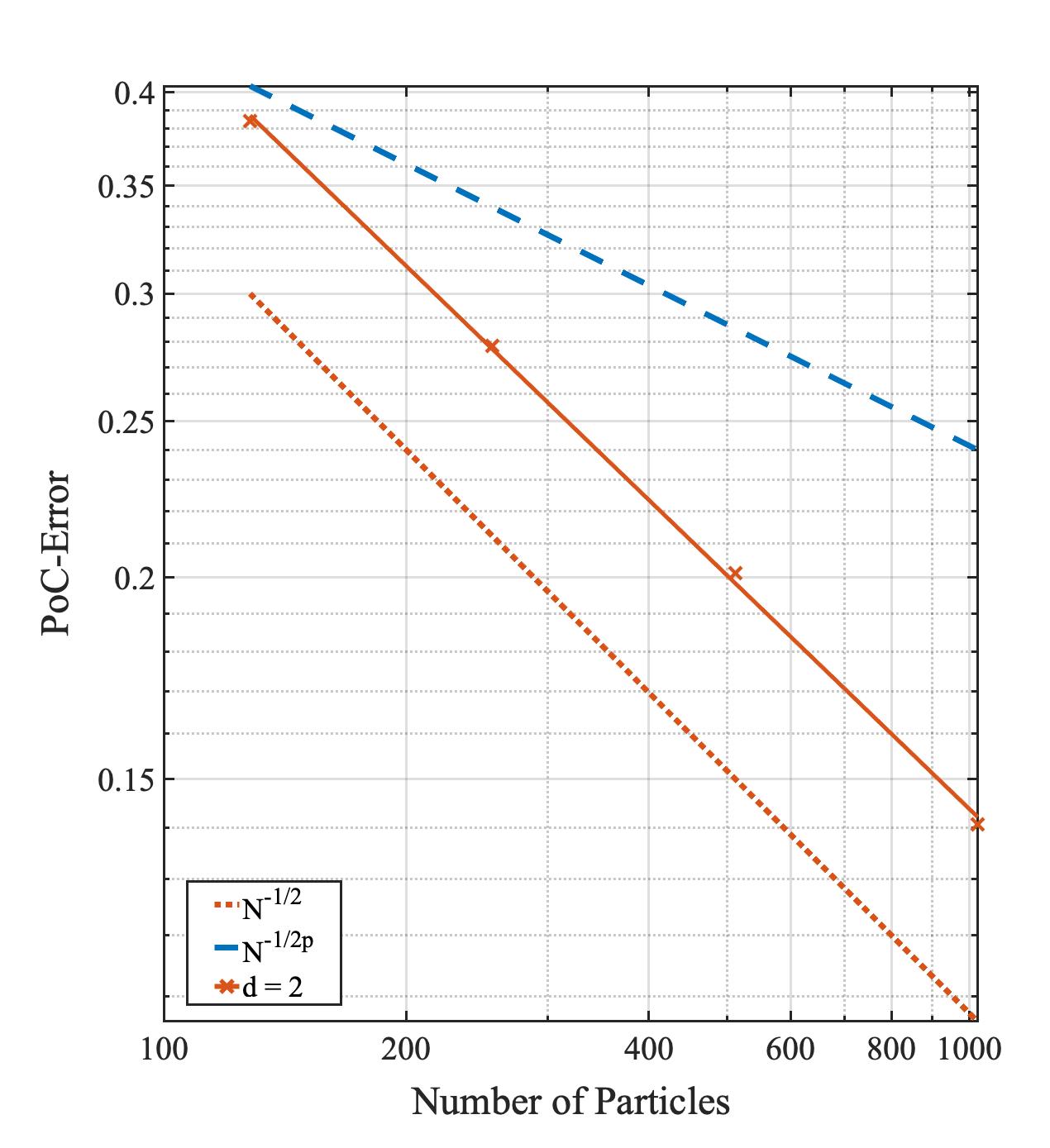}}
\subfigure[d = 4]{
\label{figure_4.2}		
\includegraphics[width=0.32\textwidth]{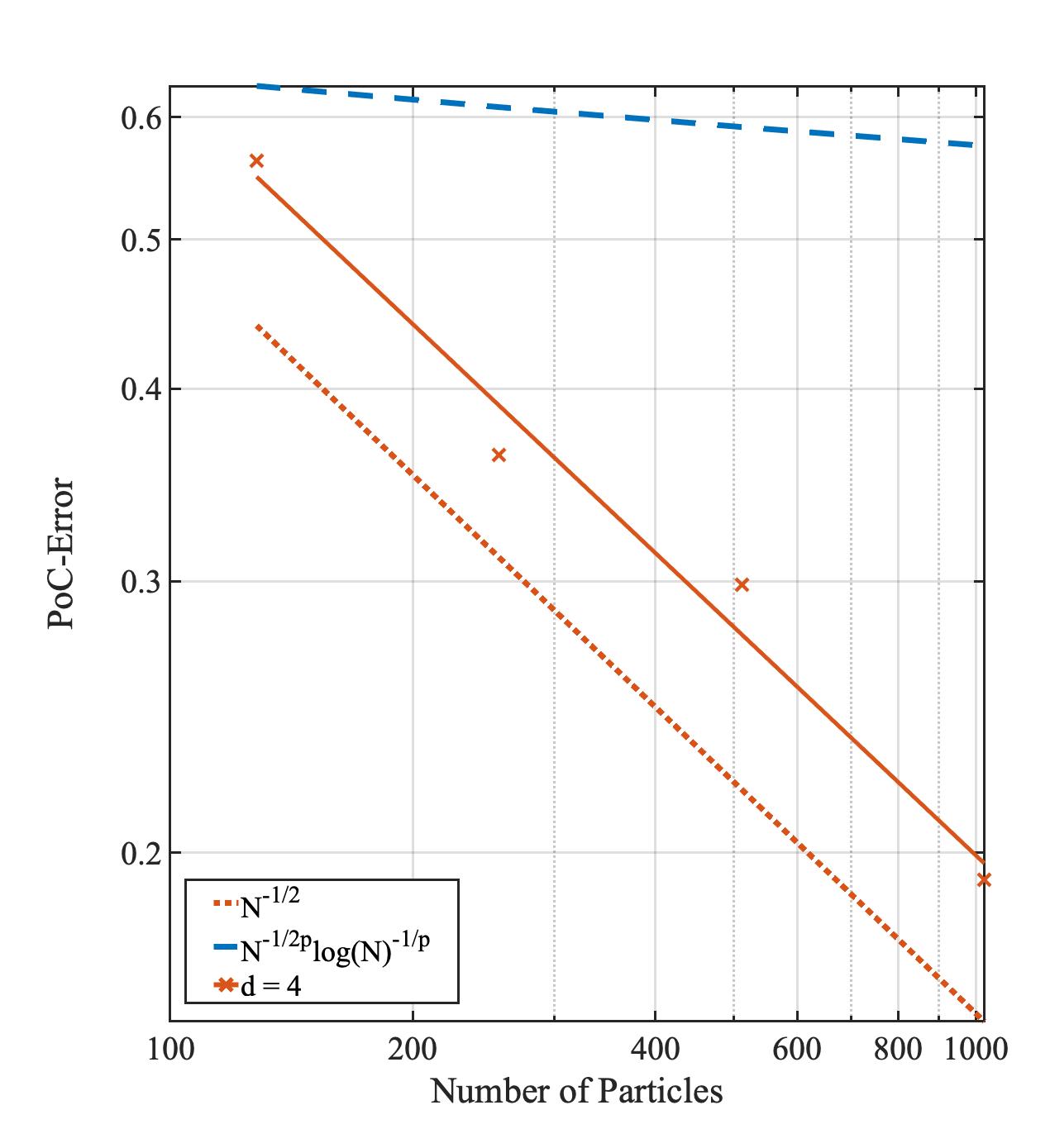}}
\subfigure[d = 6]{
\label{figure_4.3}		
\includegraphics[width=0.32\textwidth]{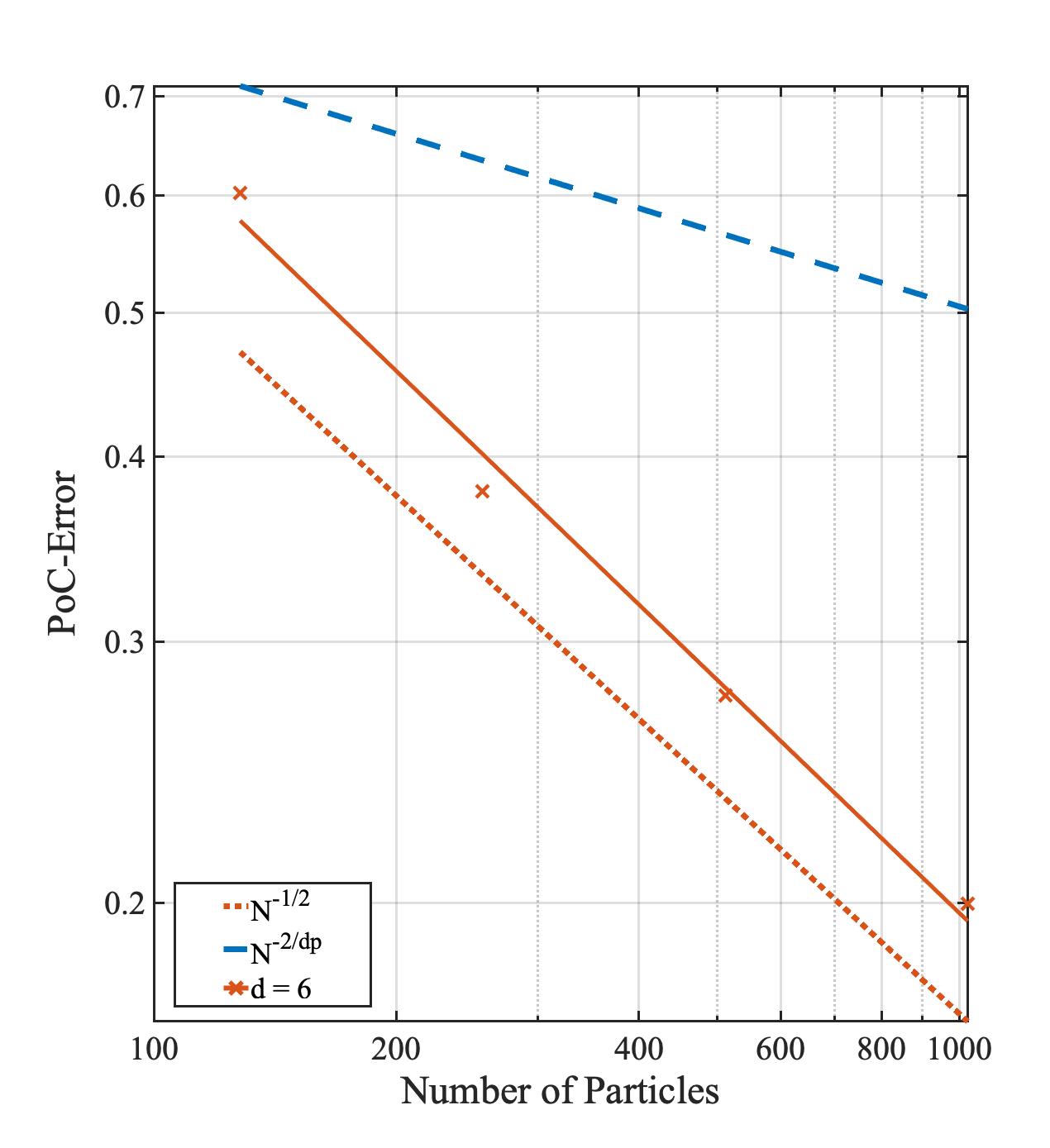}}
\caption{$L^2$ convergence rate of PoC for Example \ref{exam3}.}
\label{Fig.main4}
\end{figure}

\begin{figure}
\centering  
\subfigure[d = 6]{
\label{figure_4.1}		
\includegraphics[width=0.32\textwidth]{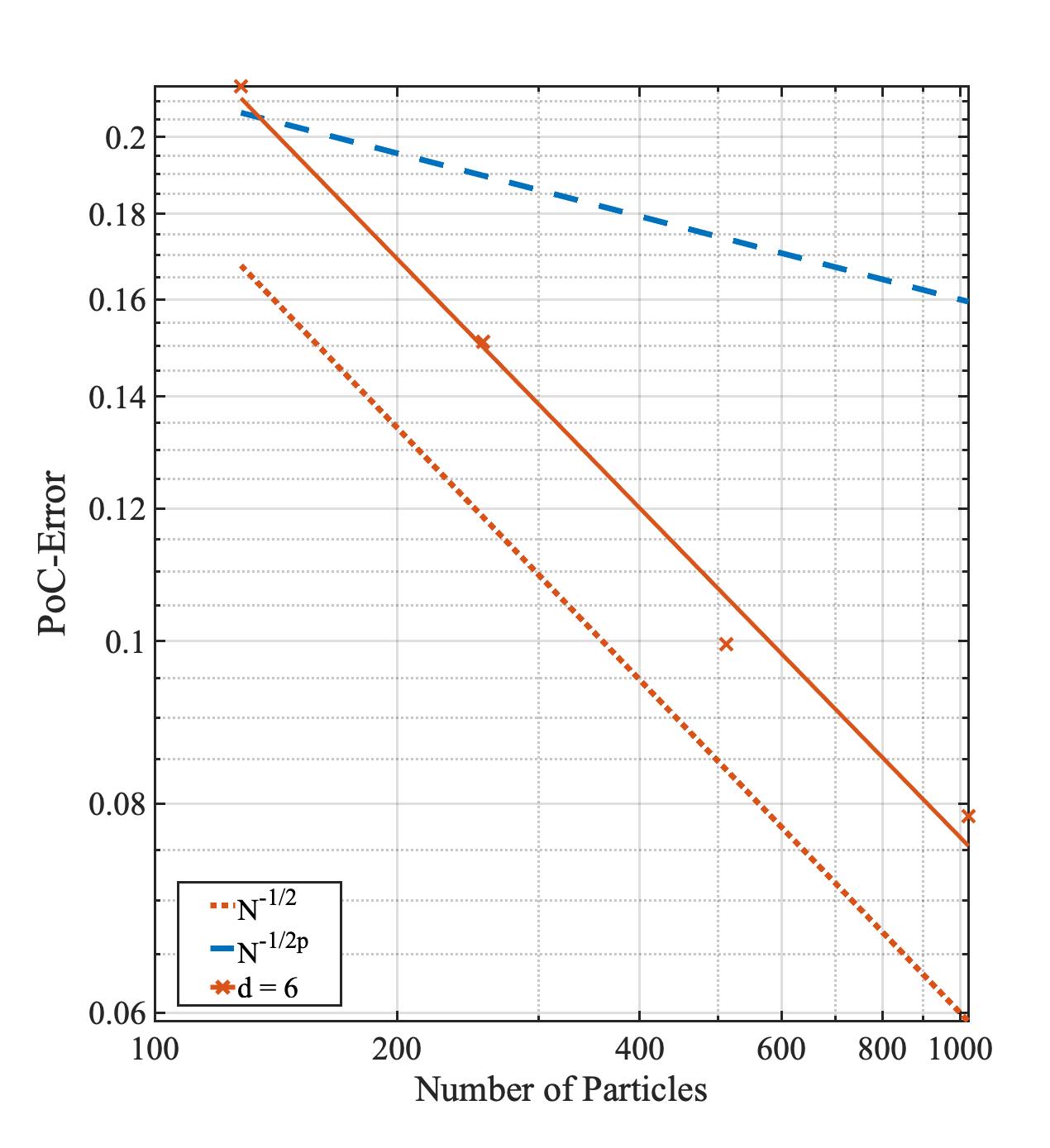}}
\subfigure[d = 8]{
\label{figure_4.2}		
\includegraphics[width=0.32\textwidth]{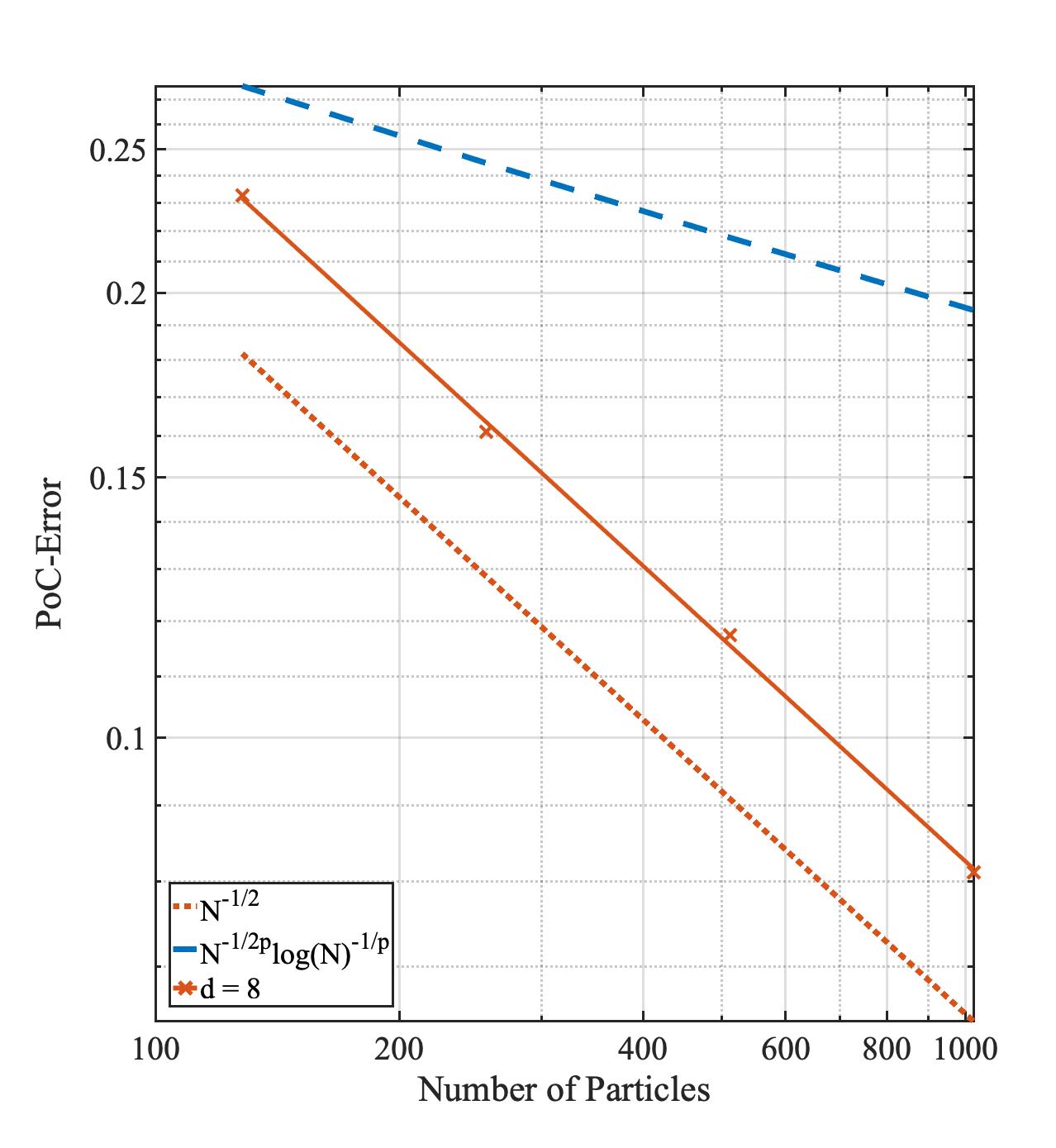}}
\subfigure[d = 10]{
\label{figure_4.3}		
\includegraphics[width=0.32\textwidth]{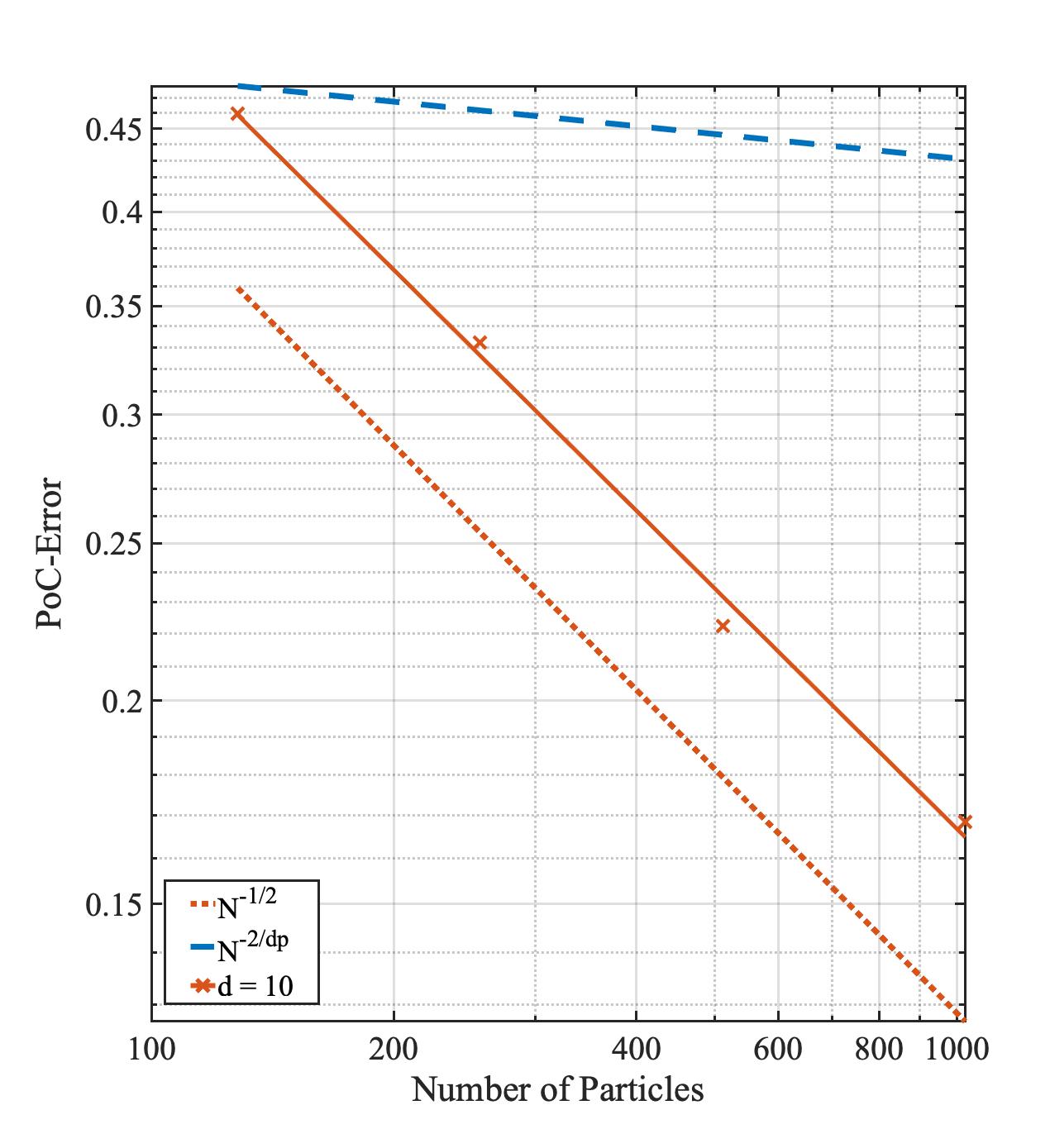}}
\caption{$L^4$ convergence rate of PoC for Example \ref{exam3}.}
\label{Fig.main5}
\end{figure}

\begin{example}\label{exam3}\rm
In this example, we test the rate of PoC for MV-SDEs with the more general nonlinear distribution-dependence coefficients. Consider the $d$-dim nonlinear MV-SDE \eqref{NIPS} with coefficients $a(x)=x-x^5$, $A(x,y)=\frac{1}{1+{\rm e}^{-x-y}}$, $B(x,y)=\sqrt{x^2+y^2}$, $\kappa_1(x,y)=\arctan(x+y)$, $\kappa_2(x,y)=\arctan(x-y)$,
\begin{align*}
 \zeta_1(x,y)=&\left(\begin{array}{cccc}\vert x_1+y_1\vert & x_2& \cdots & x_d\\x_1 & \vert x_2+y_2\vert & \cdots & x_d \\ \vdots & \vdots &  & \vdots\\x_1 & x_2 & \cdots & \vert x_d+y_d\vert\end{array}\right),\quad  \zeta_2(x,y)=\left(\begin{array}{cccc}\vert x_1-y_1\vert & x_2& \cdots & x_d\\x_1 & \vert x_2-y_2\vert & \cdots & x_d \\ \vdots & \vdots &  & \vdots\\x_1 & x_2 & \cdots & \vert x_d-y_d\vert\end{array}\right)
\end{align*} 
~\\
let $x_0\sim N(0,1)$, the the tamed EM scheme is given by  
\begin{align*}
X_{t_{n+1}}^{i,N}=&X_{t_n}^{i,N}+\frac{X_{t_n}^{i,N}-(X_{t_n}^{i,N})^5}{1+\sqrt{\Delta}\vert X_{t_n}^{i,N}-(X_{t_n}^{i,N})^5\vert }\Delta+ \left(1+\exp\left\{-\frac{1}{N}\sum_{j=1}^N\left(\kappa_1(X_{t_n}^{i,N},X_{t_n}^{j,N})-\kappa_2(X_{t_n}^{i,N},X_{t_n}^{j,N})\right)\right\}\right)^{-1}\Delta\\
&+ \sqrt{\left(\frac{1}{N}\sum_{j=1}^N\zeta_1(X_{t_n}^{i,N},X_{t_n}^{j,N})\right)^2+\left(\frac{1}{N}\sum_{j=1}^N\zeta_2(X_{t_n}^{i,N},X_{t_n}^{j,N})\right)^2}\Delta W_n^i.
\end{align*}
The orders of convergence with respect to the number of particles $N$ are shown in Figures \ref{Fig.main4} and \ref{Fig.main5}. The proxy solution takes $N=2^{11}$. The numerical solutions take $\bar{N} =2^7, 2^8, 2^9, 2^{10}$.

\end{example}

\begin{example}\label{exam4}\rm
Finally, we test the PoC rate for MV-SDEs with coefficients contain the 2nd order interaction terms. Constrained by the computational capability of our computers, here we only give numerical verification that the rate of convergence can reach $1/2$ for the following 1-dimensional equation. Let $a(x)=x-x^5$, $A(x)=\tanh(x)$, $B(x)=\frac{1}{1+{\rm e}^{-x}}$, $\kappa(x,y,z)=\vert x+y+z\vert$, $ \zeta(x,y,z)=\frac{x+y}{\sqrt{1+x^2+y^2+z^2}}$.
let $x_0\sim N(0,1)$, the tamed EM scheme is given by  
\begin{align*}
X_{t_{n+1}}^{i,N}=&X_{t_n}^{i,N}+\frac{X_{t_n}^{i,N}-(X_{t_n}^{i,N})^5}{1+\sqrt{\Delta}\vert X_{t_n}^{i,N}-(X_{t_n}^{i,N})^5\vert }\Delta+\tanh\left(\frac{1}{N^2}\sum_{j,r=1}^N \vert X_{t_n}^{i,N}+X_{t_n}^{j,N}+X_{t_n}^{r,N}\vert\right)\Delta\\
&+  \left(1+\exp\left\{-\frac{1}{N^2}\sum_{j,r=1}^N\frac{X_{t_n}^{i,N}+X_{t_n}^{j,N}}{\sqrt{1+(X_{t_n}^{i,N})^2+(X_{t_n}^{j,N})^2+(X_{t_n}^{r,N})^2}}\right\}\right)^{-1}\Delta W_n^i.
\end{align*}
The orders of convergence with respect to the number of particles $N$ are shown in Figure \ref{Fig.main6}. The proxy solution takes $N=2^{8}$. The numerical solutions take $\bar{N} =2^4, 2^5, 2^6, 2^7$.
 \begin{figure}[H]
\centering  
\includegraphics[width=0.5\textwidth]{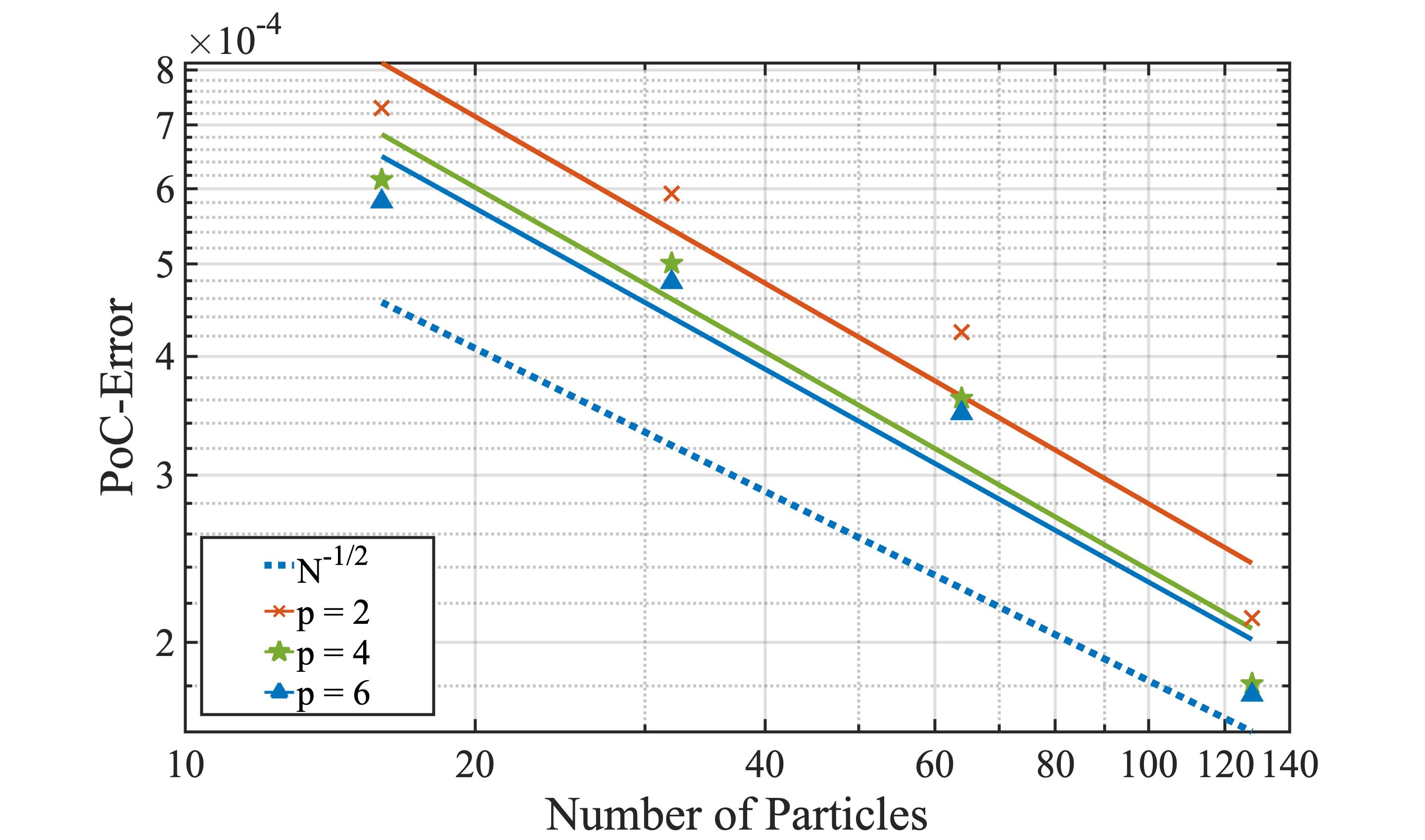}
\caption{$L^p$ convergence rate of PoC for Example \ref{exam4}.}
\label{Fig.main6}
\end{figure}

\end{example}

\section{Discussion}\label{Discussion}
The main contribution of this paper is that we theoretically prove that the PoC convergence in the sense of $L^p, p\ge2$ is of order $\mathcal{O}(N^{-1/2})$, which is dimension-independent, for several class of McKean-Vlasov stochastic differential equations (MV-SDEs) with coefficients non-linear dependence on the measure component. Both theoretical and numerical results show that this rate of convergence is across dimensions. 
However, the equations considered in this work, while nonlinear in their dependence on the measure, still necessitate satisfying the Lipschitz conditions. When the dependence of the coefficients on the measure is highly nonlinear, the relevant research still deserves further exploration. In addition, an off-topic note, in view of the computational time and cost, the particle approximation does not seem to be a good choice when the equations contain higher-order interaction terms, and it would be more important to explore the numerical schemes with lower computational cost.

\section*{Acknowledgments}
Yuhang Zhang is supported by the China Postdoctoral Science Foundation under Grant Number 2024M754160 and the Postdoctoral Fellowship Program of CPSF under Grant Number GZC20242217.
Minghui Song is supported in part by funds from the National Natural Science Foundation of China under Grant Numbers 12471372 and 12071101.

\end{document}